\documentclass[preprint,3p,times]{elsarticle}
\usepackage{array,multirow,bm}
\usepackage{siunitx} 
\usepackage{enumerate} 
\usepackage{mathrsfs}
\usepackage{amsthm,amsmath,amssymb} 
\usepackage[normalem]{ulem}
\usepackage{graphicx,subcaption}
\usepackage{algorithm,algorithmicx,algpseudocode}
\usepackage{hyperref}
\hypersetup{colorlinks=true,
linkcolor=blue,
anchorcolor=blue,
citecolor=blue}

\allowdisplaybreaks[4]

\newtheorem{theorem}{Theorem}[section]
\newtheorem{lemma}[theorem]{Lemma}
\newtheorem{assumption}[theorem]{Assumption}
\newtheorem{remark}{Remark}[section]
\numberwithin{equation}{section}
\numberwithin{figure}{section}
\numberwithin{table}{section}



\journal{XXX}

\begin{document}

\begin{frontmatter}

\title{A low-rank solver for the Stokes–Darcy model with random hydraulic conductivity and Beavers–Joseph condition}

\author[addr1]{Yujun Zhu}
\ead{yujun_zhu@hust.edu.cn}

\author[addr1]{Yulan Ning}
\ead{ylning@hust.edu.cn}

\author[addr2]{Zhipeng Yang} 
\ead{yangzp@lzu.edu.cn}

\author[addr3]{Xiaoming He} 
\ead{hex@mst.edu}

\author[addr1,addr4]{Ju Ming} 
\ead{jming@hust.edu.cn}

\affiliation[addr1]{organization={School of Mathematics and Statistics, Huazhong University of Science and Technology},
	city={Wuhan},
	postcode={430074},
	state={Hubei},
	country={P.R.China}}

\affiliation[addr2]{organization={School of Mathematics and Statistics, Lanzhou University},
	city={Lanzhou},
	postcode={730000},
	state={Gansu},
	country={P.R.China}}

\affiliation[addr3]{organization={Department of Mathematics and Statistics, Missouri University of Science and Technology},
	city={Rolla},
	postcode={65409},
	state={MO},
	country={USA}}

\affiliation[addr4]{organization={Hubei Key Laboratory of Engineering Modeling and Scientific Computing, Huazhong University of Science and Technology},
	city={Wuhan},
	postcode={430074},
	state={Hubei},
	country={P.R.China}}

\begin{abstract}
	
	This paper proposes, analyzes, and demonstrates an efficient low-rank solver for the stochastic Stokes-Darcy interface model with a random hydraulic conductivity both in the porous media domain and on the interface. We consider three interface conditions with randomness, including the Beavers–Joseph interface condition with the random hydraulic conductivity, on the interface between the free flow and the porous media flow. Our solver employs a novel generalized low-rank approximation of the large-scale stiffness matrices, which can significantly cut down the computational costs and memory requirements associated with matrix inversion without losing accuracy. Therefore, by adopting a suitable data compression ratio, the low-rank solver can maintain a high numerical precision with relatively low computational and space complexities. We also propose a strategy to determine the best choice of data compression ratios. Furthermore, we carry out the error analysis of the generalized low-rank matrix approximation algorithm and the low-rank solver. Finally, numerical experiments are conducted to validate the proposed algorithms and the theoretical conclusions.
	
\end{abstract}


\end{frontmatter}

\section{Introduction}\label{sec1}

The Stokes–Darcy interface model is a fundamental model for coupling the incompressible flow and the porous media flow with the interface conditions. In the past two decades, it has become a popular research area in computational fluid dynamics, due to its wide applications in industrial and engineering problems, including interaction between surface and subsurface flows \cite{ccecsmeliouglu2009primal, MDiscacciati_1, dong2013hybrid, furman2008modeling, layton2013analysis, WJLayton_FSchieweck_IYotov_1},  groundwater system in karst aquifers \cite{ghasemizadeh2012groundwater, hu2012experimental, neven2021modeling}, oil reservoir in vuggy porous medium \cite{arbogast2007computational, arbogast2006homogenization, YGao_XMHe_LMei_XYang_1,
JHou_MQiu_XMHe_CGuo_MWei_BBai_1}, and industrial filtrations \cite{ervin2009coupled,hanspal2006numerical}. However, the sophisticated coupled structure leads to various major difficulties, and thus many attempts have been made to seek promising numerical methods for solving the deterministic coupled model, such as the domain decomposition methods \cite{YBoubendir_STlupova_2, YCao_MGunzburger_XMHe_XWang_2, chen2011parallel, MDiscacciati_LGerardo-Giorda_1,
MDiscacciati_PGervasio_AGiacomini_AQuarteroni_1,
MDiscacciati_EMiglio_AQuarteroni_1, MDiscacciati_AQuarteroni_AValli_1,  MGunzburger_XMHe_BLi_1,YLiu_YBoubendir_XMHe_YHe_1, STlupova_1}, multi-grid methods \cite{arbogast2009discretization, cai2009numerical, mu2007two}, and discontinuous Galerkin methods \cite{VGirault_BRiviere_1, GKanschat_BRiviere_1, RLi_YGao_JChen_LZhang_XMHe_ZChen_1, lipnikov2014discontinuous}.

Due to the difficulties in physically measuring its exact values, the hydraulic conductivity tensor in porous media is generally determined via statistical learning in practical applications in material science, geophysics and chemical engineering, such as linear regressions \cite{roding2020predicting}, Bayesian ridge regressions \cite{tian2022improved} and data-driven surrogate models \cite{graczyk2020predicting,takbiri2022deep}. Consequently, randomness needs to be considered to obtain more realistic and reliable results. A popular idea is to model the hydraulic conductivity by a perturbation form that consists of a deterministic function and a finite-dimensional Gaussian noise, for instance, the Karhunen-Loe\.{v}e expansions \cite{betz2014numerical, loeve1977elementary} and the polynomial chaos expansions \cite{xiu2002modeling, xiu2002wiener}. Then the original deterministic partial differential equation (PDE) is recast as a stochastic system, and we focus on the probabilistic information of the solutions, i.e., their statistical moments. Various efficient numerical methods have been proposed to solve stochastic partial differential equations (SPDEs), such as the Monte Carlo (MC) method and its variants \cite{ali2017multilevel, barth2011multi, cliffe2011multilevel,fishman2013monte, gunzburger2014stochastic, 
guth2021quasi, helton2003latin, niederreiter1992random,novak1988stochastic, yang2022multigrid}, the polynomial chaos methods \cite{jakeman2017generalized, xiu2002modeling, xiu2002wiener}, the stochastic collocation methods \cite{babuvska2007stochastic, babuvska2010stochastic,guo2017stochastic}, the sparse grid methods \cite{bao2014hybrid, nobile2008anisotropic, nobile2008sparse,yang2021stochastic}, and the stochastic Galerkin methods \cite{babuska2004galerkin,ghanem2003stochastic,leykekhman2012investigation,li2018discontinuous,sun2016priori}. 

The MC method \cite{metropolis1949monte} is one of the most popular techniques in statistical estimation owing to its convergence rate independent of the dimension of the probabilistic space. However, a large number of MC samples are required to obtain reliable statistical moments by the Central Limit Theorem, and its application to stochastic partial differential equations on a fine mesh results in large degrees of freedom and high-dimensional stiffness matrices, which significantly increase computational and memory loads. To mitigate this challenge, dimensionality reduction techniques offer a promising approach to handle these large-scale linear systems.

Dimensionality reduction aims to produce a more compact data representation with lower dimensions while achieving an acceptable loss of the intrinsic information. The vector space model \cite{turk1991eigenfaces, zhao2003face} is one of the most universally adopted dimensionality reduction approaches, which relies on vector selection. In this framework, each coefficient matrix is interpreted as a sequence of column vectors. An important usage of the vector space model is to generate the low-rank approximation of matrices by using singular value decomposition (SVD), which achieves the smallest reconstruction error in the Euclidean distance among all approximations with the rank constraint \cite{eckart1936approximation}. Despite its appealing numerical accuracy in data reconstruction, the standard SVD method becomes computationally prohibitive when applied to large-scale matrices.

To conquer the limitations in computational cost and memory usage, the idea of ``generalized" has been introduced to low-rank approximation \cite{ye2004generalized}. It aims to design a dimensionality reduction approach for a collection of matrices, where the two-sided orthogonal matrices are fixed and shared by all initial matrices. Many variants of the generalized low-rank approximation of matrices (GLRAM) have been developed to further enhance the dimensionality reduction performance. The non-iterative GLRAM method in \cite{liu2006non} used an analytical formulation instead of the iterative manner, and the simplified GLRAM method in \cite{lu2008simplified} simplified the standard structure, which both significantly mitigated the computational burden. The robust GLRAM approach demonstrated strong resilience to large sparse noise and outliers \cite{shi2015robust, zhao2016robust}. The application of low-rank approximation in solving the stochastic elliptic equation and its corresponding optimal control problem was explored in \cite{zhu2024low}. However, the current research mainly focuses on dimensionality reduction in computer vision and image processing \cite{berry1995using, LI2022251, turk1991eigenfaces, ye2004generalized, zhao2003face}, and there remains a notable gap in the literature concerning its application in solving nonlinear or coupled PDE systems.

In this article, we develop an efficient low-rank solver for the Stokes-Darcy interface model with a random hydraulic conductivity and three interface conditions, including the Beavers–Joseph interface condition with the random hydraulic conductivity. Compared with the previous work \cite{zhu2024low}, which considered the elliptic partial differential equations with random diffusion coefficient in the total domain, the randomness of the Stokes-Darcy interface model originates from the porous media domain and is transmitted into the conduit domain through the interface conditions. Unlike the construction of simple random variables in \cite{zhu2024low}, to obtain more realistic and reliable quantities of interest (QoIs), the hydraulic conductivity is perturbed by a finite-dimensional Gaussian field with zero mean, and the Karhunen-Loe\.{v}e expansion is utilized and truncated to approximate the random coefficient. Specifically, we first propose a novel generalized low-rank matrix approximation method, which performs well in balancing the numerical precision, the computational cost and memory requirement. Then we apply the dimension reduction technique to deal with the high-dimensional stiffness matrices from the stochastic coupled model. The Sherman-Morrison-Woodbury formula is also adopted to obtain the inverse of these matrices. The schematic flowchart of the low-rank solver is illustrated in Figure \ref{fig1.1}. By choosing a suitable data compression ratio, our low-rank solver can achieve significant improvements in computational efficiency and storage load without losing much accuracy. We also design a strategy to determine a good data compression ratio, which was not completed in \cite{zhu2024low}. The error analyses of our generalized low-rank approximation of matrices and the low-rank solver are conducted. And we carry out numerical experiments to demonstrate the numerical performance and confirm the theoretical results of the proposed algorithms.

\begin{figure}[htbp]

\centering
\includegraphics[width=1\textwidth]{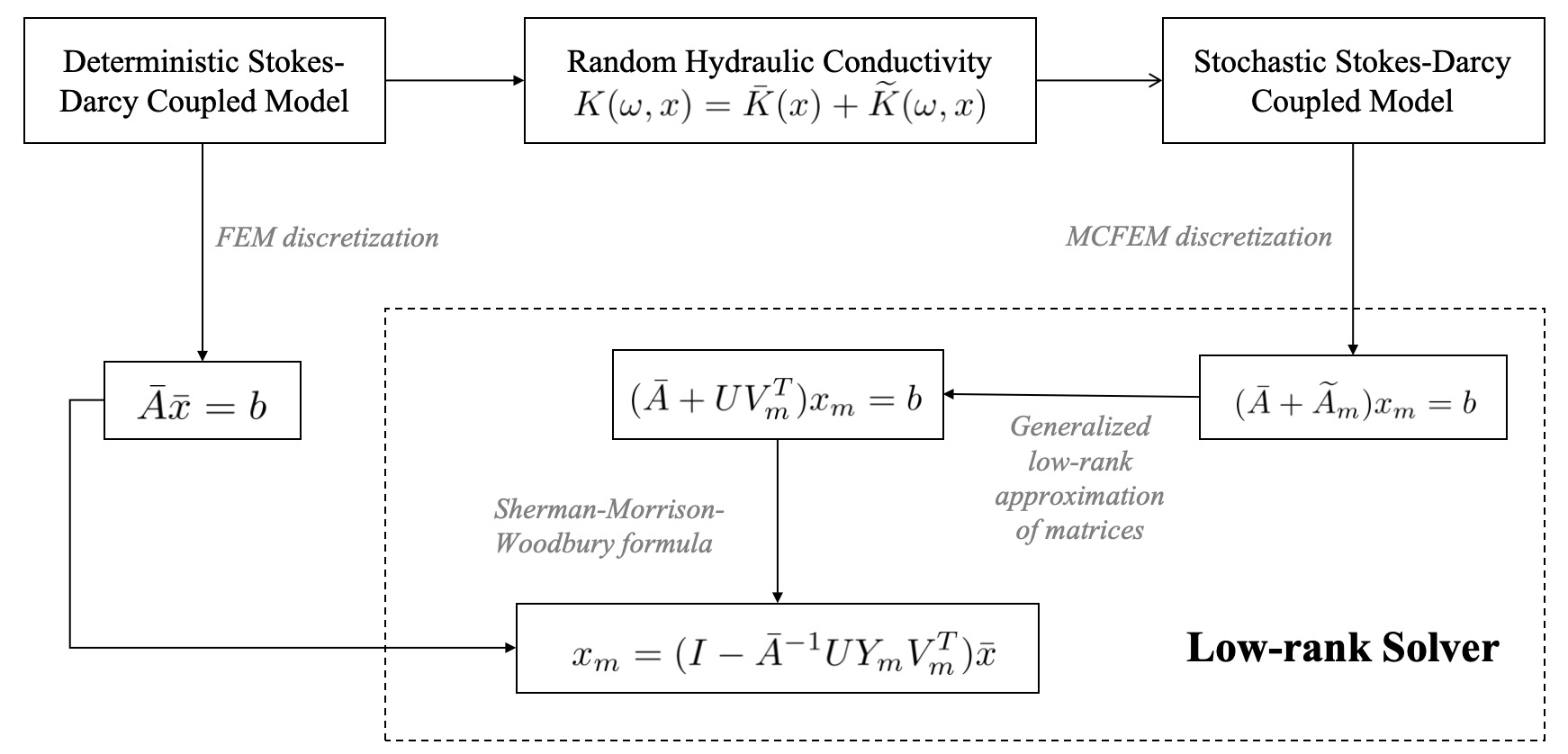}
\caption{The flowchart of the low-rank solver to the Stokes-Darcy coupled model with random hydraulic conductivity.}
\label{fig1.1}
\end{figure}

The remainder of this article is structured as follows. Section \ref{sec2} provides a brief overview of the relevant function spaces and notations. In Section \ref{sec3}, we introduce the stochastic Stokes-Darcy interface model and its weak formulation. The random hydraulic conductivity has a perturbarion form and we recall the Monte Carlo finite element method to compute the numerical solutions of the stochastic system. In Section \ref{sec4}, a novel generalized low-rank matrix approximation technique is proposed and we adopt it in solving SPDEs. Then, we design an efficient low-rank solver for the stochastic coupled model and provide its error analysis. Section \ref{sec5} presents the numerical results that illustrate the nice features of the proposed solver. Conclusions can be found in Section \ref{sec6}.

\section{Preliminaries}\label{sec2}

We begin by introducing the necessary function spaces and notations. Throughout this article, we adopt the standard definitions for Sobolev spaces from \cite{adams2003sobolev}. Let $D \subset \mathbb{R}^d$ for $d = 2, 3$ be an open, connected, bounded, and convex domain with polygonal and Lipschitz continuous boundary $\partial D$. Let $L^p(D),  1 \leq p \leq \infty$ be an usual Lebesgue space on $D$, with the $L^2(D)$-norm $\Vert \cdot \Vert = \Vert \cdot \Vert_{L^2(D)}$ induced by the inner product 
$$(f, g)  = \int_D fg dx, \, \forall f,g\in L^2(D).$$ 
Let $W^{r,q}(D), r \in \mathbb{Z}^+, 1 \leq q \leq \infty$ denote a deterministic Sobolev space on $D$. For $q=2$, define the Hilbert space $H^r(D) :=W^{r,2}(D)$ equipped with the following standard norm and semi-norm,
\begin{equation} \label{eq2.1}
\Vert y \Vert_{H^r(D)} = \displaystyle \sum_{\vert \bm{\alpha} \vert \leq r} \Vert \frac{\partial^{\bm{\alpha}} y}{\partial y^{\bm{\alpha}}} \Vert^2_{L^2(D)},
\end{equation}

\begin{equation} \label{eq2.2}
\vert y \vert_{H^r(D)} = \displaystyle \sum_{\vert \bm{\alpha} \vert = r} \Vert \frac{\partial^{\bm{\alpha}} y}{\partial y^{\bm{\alpha}}} \Vert^2_{L^2(D)}.
\end{equation}

\noindent where $\alpha$ is a multi-index with non-negative integer components $\{\alpha_i\}$, and $\bm{\alpha} = \displaystyle \sum_i \alpha_i$. 

Denote the Hilbert spaces

\begin{equation} \label{eq2.3}
H^1(D) = \{ y \in L^2(D), \partial_{x_i} y \in L^2(D), i = 1,...,n\},
\end{equation}

\noindent and

\begin{equation} \label{eq2.4}
H_0^1(D) = \{ y \in H^1(D), y \vert_{\partial D} = 0\}.
\end{equation}

\noindent Clearly, $H_0^1(D) \subset H^1(D)$ has the dual space $H^{-1}(D)$ . For d = 2, 3, we define $\bm{L}^p(D) := (L^p(D))^d$ and $\bm{H}^r(D) := (H^r(D))^d$.

Let $(\Omega, \mathscr{F}, \mathcal{P})$ be a complete probability space, where $\Omega$ is the set of outcomes, $\mathscr{F} \subset 2^\Omega$ is the $\sigma$-algebra of events, and $\mathcal{P}: \mathscr{F} \rightarrow [0,1]$ is a complete probability measure. For a real random variable $X \in (\Omega, \mathscr{F}, \mathcal{P})$, its expectation is given by

\begin{equation} \label{eq2.5}
\mathbb{E}[X] = \int_\Omega X(\omega) \mathcal{P}(\mathrm{d} \omega) = \int_{\mathbb{R}^n} x \rho(x) \mathrm{d}x,
\end{equation}

\noindent where $\rho$ denotes the joint probability density function of $X$, defined on a Borel set $\mathrm{B}$ of $\mathbb{R}$, that is, $\rho(\mathrm{B}) = \mathcal{P}(X^{-1}(\mathrm{B}))$.

Define the stochastic Sobolev space as

\begin{equation} \label{eq2.6}
L^2(\Omega; H^r(D)) = \{ y: D \times \Omega \rightarrow  H^r(D) \; \vert \; \Vert y \Vert_{L^2(\Omega; H^r(D))}  <  \infty \},
\end{equation}

\noindent equipped with the norm 

\begin{equation} \label{eq2.7}
\Vert y \Vert_{L^2(\Omega; H^r(D))} = \int_\Omega \Vert y \Vert_{H^r(D)} d \mathcal{P} = \mathbb{E}[\Vert y \Vert_{H^r(D)}].
\end{equation}

The stochastic Sobolev space $L^2(\Omega; H^r(D))$ is a Hilbert space, and it is isomorphic to the tensor product space $L^2(\Omega) \otimes L^2(H^r(D))$. For convenience, we introduce the following notations

\begin{equation} \label{eq2.8}
\begin{aligned}
	\mathcal{L}^p(D) &= L^2(\Omega; L^p(D)), \: with \: norm\: \Vert \cdot \Vert_{\mathcal{L}^p(D)} =  \Vert \cdot \Vert_{L^2(\Omega; L^p(D))},\\
	\mathcal{H}^r(D) &= L^2(\Omega;H^r(D)), \: with \: norm\: \Vert \cdot \Vert_{\mathcal{H}^r(D)} =  \Vert \cdot \Vert_{L^2(\Omega;H^r(D))},\\
	\bm{\mathcal{L}}^p(D) &= L^2(\Omega; \bm{L}^p(D)), \: with \: norm\: \Vert \cdot \Vert_{\bm{\mathcal{L}}^p(D)} =  \Vert \cdot \Vert_{L^2(\Omega; \bm{L}^p(D))},\\
	\bm{\mathcal{H}}^r(D) &= L^2(\Omega; \bm{H}^r(D)), \: with \: norm\: \Vert \cdot \Vert_{\bm{\mathcal{H}}^r(D)} =  \Vert \cdot \Vert_{L^2(\Omega; \bm{H}^r(D))}.
\end{aligned}
\end{equation}

\section{Stokes-Darcy interface model with random permeability} \label{sec3}

The Stokes–Darcy system describes the underground flow in porous media governed by the Darcy equation and the incompressible fluid flow governed by the Stokes equation, while these two flows are coupled on the interface between the conduits and porous media. To overcome the difficulty of determining the exact hydraulic conductivity tensor in the Stokes-Darcy model, we consider a coupled stochastic system where the uncertainties originally come from the porous domain and then affect the conduit domain through the interface. In this section, we present the stochastic model and recall the basic Monte Carlo finite element method for solving this model.

\subsection{The stochastic Stokes–Darcy interface model} \label{subsec3.1}

In a bounded smooth domain $D \in \mathbb{R}^d, d = 2, 3$, we consider the coupled Stokes-Darcy system of the confined flow in the porous medium region $D_p$, and the free flow in the conduit region $D_f$. For simplicity, we assume that $\partial D_p$ and $\partial D_f$ are sufficiently smooth in the work. Here $D = D_p \cup D_f$, and these two subdomains lie across the interface $\Gamma_I = \partial D_p \cap \partial D_f$. The outer boundary $\partial D$ is decomposed into two disjoint components: $\Gamma_p =\partial D_p \verb|\| \Gamma_I$ and $\Gamma_f =\partial D_f \verb|\| \Gamma_I$. Let $n_p, n_f$ denote the unit outer normal to $D_p$ and $D_f$ respectively, and $\bm{\tau}_j, j = 1,..., d-1$ denote a set of mutually orthogonal unit tangential vectors to the interface $\Gamma_I$. Note that $\bm{n}_p = -\bm{n}_f$ on $\Gamma_I$, see Figure \ref{fig3.1} for a two-dimensional sketch.

\begin{figure}[htbp]

\centering
\includegraphics[width=0.55\textwidth]{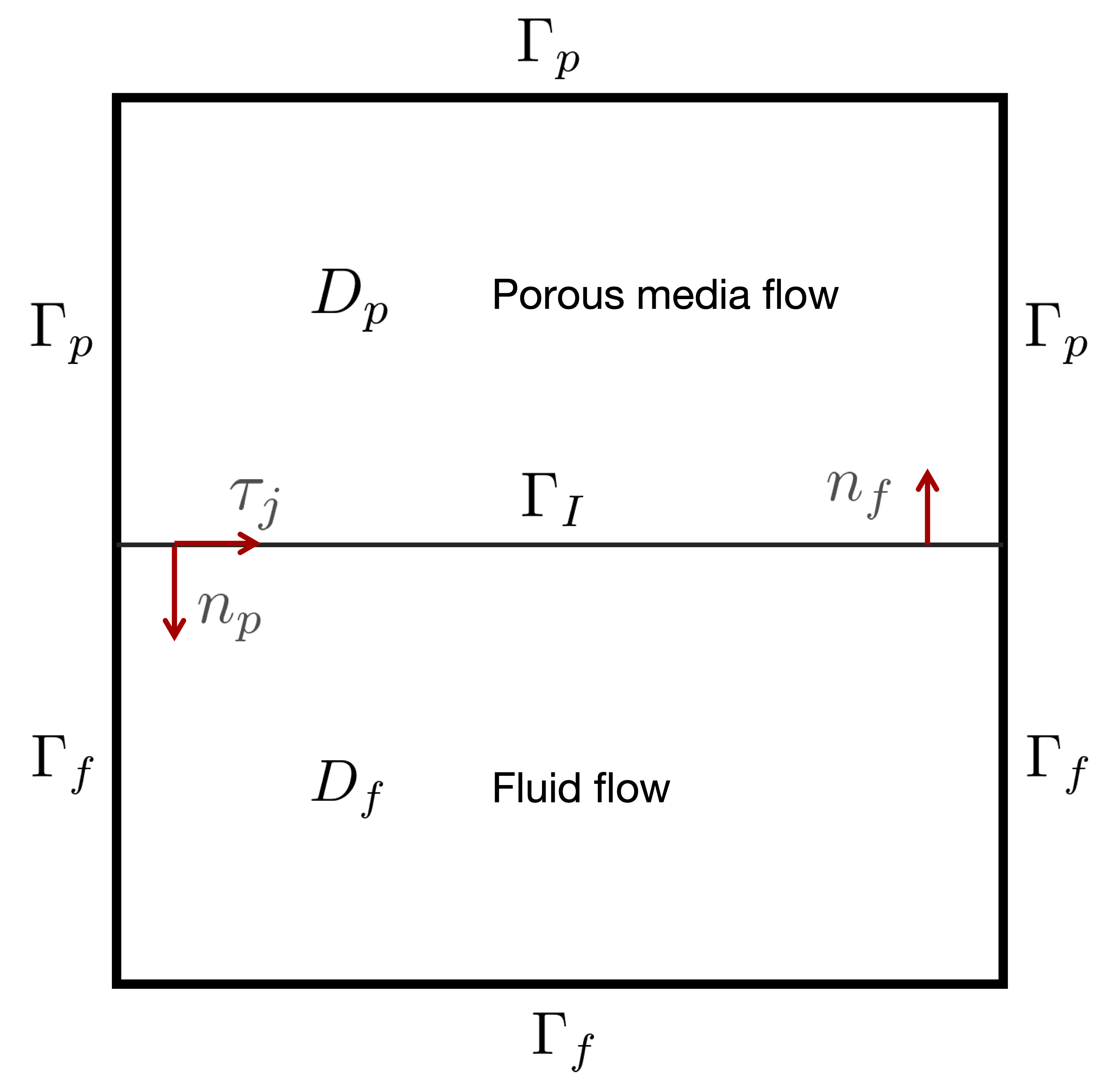}
\caption{A two-dimensional sketch of the porous media domain $D_p$, the conduit domain $D_f$, and the interface $\Gamma_I$.}
\label{fig3.1}
\end{figure}

The fluid motion in the porous media domain $D_p$ is governed by the Darcy equations \cite{barenblatt1960basic} as below
\begin{subequations}
\begin{align}
	\bm{u}_p(\omega,x)  &  =  -\mathbb{K}(\omega,x) \nabla \phi_p(\omega,x), \quad \text{in} \enspace D_p, \label{Za}\\
	\nabla \cdot \bm{u}_p(\omega,x)  &= f_p(x),\quad \text{in} \enspace D_p, \label{Zb}
\end{align}
\end{subequations}

\noindent where $\bm{u}_p$ denotes the flow velocity in the porous media, which is proportional to the gradient of $\phi_p$. $\phi_p$ is the hydraulic head and $f_p$ is the deterministic sink/source term. The uncertainties in the coupled system arise from the hydraulic conductivity tensor of the porous media $\mathbb{K}(\omega,x) \; (\omega \in \Omega, x \in \bar{D}_p$), which is defined on the complete probability space $(\Omega, \mathscr{F}, \mathcal{P})$ and assumed to be symmetric and positive definite. By substituting (\ref{Za}) into (\ref{Zb}), we reach the second-order form of the Darcy equation:

\begin{equation} \label{eq3.2}
- \nabla \cdot (\mathbb{K}(\omega,x) \nabla \phi_p(\omega,x))= f_p(x),\quad \text{in} \enspace D_p.
\end{equation}

The uncertainties in (\ref{eq3.2}) affect the solutions in the conduit domain $D_f$ through the interface $\Gamma_I$. Therefore, the incompressible fluid flow is modeled by the following stochastic Stokes equations \cite{arnold1984stable}

\begin{equation} \label{eq3.3}
\left\{
\begin{aligned}
	- \nabla \cdot \mathbb{T}(\bm{u}_f(\omega,x), p_f(\omega,x)) &= \bm{f}_f(x), \quad \text{in} \enspace D_f,\\
	\nabla \cdot \bm{u}_f(\omega,x) &= 0, \quad \text{in} \enspace D_f,
\end{aligned}
\right.
\end{equation}

\noindent where $u_f$ is the specific discharge in the conduit region $D_f$, $p_f$ is the kinematic pressure, $f_f$ is the external body force, and

\begin{equation} 
\mathbb{T}(\bm{u}_f, p_f) = -p_f \mathbb{I} + 2 \nu \mathbb{D}(\bm{u}_f), \enspace \mathbb{D}(\bm{u}_f) = \frac{1}{2} (\nabla \bm{u}_f + \nabla^T \bm{u}_f),
\nonumber
\end{equation}

\noindent are the stress tensor and deformation rate tensors. Here $\nu > 0$ is the kinematic viscosity of the fluid and $\mathbb{I}$ is the identity matrix. 

On the interface $\Gamma_I$ between the porous medium $D_p$ and the free-flow region $D_f$, we consider the following three interface conditions with random hydraulic conductivity \cite{yang2022multigrid}
\begin{subequations}
\begin{align}
	\bm{u}_f(\omega,x) \cdot \bm{n}_f(x)  &  =  (\mathbb{K}(\omega,x) \nabla \phi_p(\omega,x)) \cdot \bm{n}_p(x), \quad \text{on} \enspace \Gamma_I, \label{ZZa}\\
	-\bm{n}_f^T \mathbb{T}(\bm{u}_f(\omega,x), p_f(\omega,x)) \bm{n}_f &= g(\phi_p(\omega,x) - z), \quad \text{on} \enspace \Gamma_I, \label{ZZb}\\
	- \bm{\tau}_j^T \mathbb{T}(\bm{u}_f(\omega,x), p_f(\omega,x)) \bm{n}_f &=  \frac{\alpha \nu \sqrt{d}}{\sqrt{\text{tr}(\Pi(\omega,x))}}\bm{\tau}_j^T(\bm{u}_f(\omega,x) + \mathbb{K}(\omega,x) \nabla \phi_p(\omega,x)) , \quad \text{on} \enspace \Gamma_I, \label{ZZc}
\end{align}
\end{subequations}

\noindent where $g$ is the gravitational acceleration, $z$ is the elevation head, $\Pi(\omega,x) = \frac{\mathbb{K}(\omega,x) \nu}{g}$ is the intrinsic permeability, and $\text{tr}(\cdot)$ denotes the trace of a matrix. Particularly, the last interface condition (\ref{ZZc}) represents the classical Beavers–Joseph (BJ) condition \cite{YCao_MGunzburger_XMHe_XWang_1, YCao_MGunzburger_XHu_FHua_XWang_WZhao_1, YCao_MGunzburger_FHua_XWang_1,  feng2012non, XMHe_JLi_YLin_JMing_1, mikelic2000interface, LShan_HZheng_1, yang2022multigrid}.

For simplicity, we impose the following deterministic homogeneous Dirichlet boundary conditions 

\begin{equation} \label{eq3.5}
\phi_p = 0 \enspace \text{on} \enspace \Gamma_p \quad \text{and} \quad \bm{u}_f = 0 \enspace \text{on} \enspace \Gamma_f.
\end{equation}

\subsection{Weak formulation of the coupled problem} \label{subsec3.2}

Define the following function spaces:

\begin{equation}
\begin{aligned}
	\enspace &X_p := \{\phi_p \in \mathcal{H}^1(D_p) \vert \; \phi_p = 0 \: \text{ on } \: \Gamma_p\},\\
	\enspace &X_f := \{u_f \in \bm{\mathcal{H}}^1(D_f) \vert \;u_f = 0 \: \text{ on } \: \Gamma_f\},\\
	\enspace &Q_f := \{p_f \in \mathcal{L}^2(D_f)\}.
	\nonumber
\end{aligned}
\end{equation}

For notational convenience, we define $\underline{\bm{X}} := X_p \times X_f$ and $\underline{\bm{u}} = \left( \phi_p,\bm{u}_f\right) $, where $\phi_p \in X_p, \bm{u}_f \in X_f$. The norm for $\underline{\bm{X}}$ is defined as

\begin{equation}
\Vert \underline{\bm{u}} \Vert_{\underline{\bm{X}}} = \left( \Vert \phi_p \Vert_{\mathcal{H}^1(D_p)}^2  +  \Vert \bm{u}_f \Vert_{\mathcal{H}^1(D_f)}^2\right) ^{1/2}.
\nonumber
\end{equation}

Then the weak formulation \cite{yang2022multigrid} of the stochastic Stokes-Darcy coupled model is to find $\left( \underline{\bm{u}},p_f\right) \in \underline{\bm{X}} \times Q_f$ such that

\begin{equation} \label{eq3.6}
\left\{
\begin{aligned}
	& A(\underline{\bm{u}}, \underline{\bm{v}}) - B(\underline{\bm{v}},p_f) = F(\underline{\bm{v}}), \quad \quad \,\forall \underline{\bm{v}} = \left( \psi_p,\bm{v}_f\right) \in \underline{\bm{X}},\\
	& B(\underline{\bm{u}},q_f) = 0,\qquad\qquad\qquad\; \quad \forall q_f \in Q_f,
\end{aligned}
\right.
\end{equation}

\noindent where
\begin{subequations}
\begin{align}
	& A(\underline{\bm{u}}, \underline{\bm{v}}) = \mathbb{E}\left[ a(\underline{\bm{u}}, \underline{\bm{v}}) \right] = \int_\Omega  a(\underline{\bm{u}}, \underline{\bm{v}}) \mathrm{d} \omega, \label{ZZZa}\\
	& a(\underline{\bm{u}}, \underline{\bm{v}}) = g \int_{D_p} \left(\mathbb{K} \nabla \phi_p \right) \cdot \nabla \psi_p \mathrm{d}x +  \int_{D_f}2 \nu \mathbb{D}(\bm{u}_f): \mathbb{D}(\bm{v}_f) \mathrm{d}x \nonumber\\
	& \qquad \quad \: \,+ g \int_{\Gamma_I} \phi_p  \bm{v}_f \cdot \bm{n}_f \mathrm{d}x + \int_{\Gamma_I} \frac{\alpha \nu \sqrt{d}}{\sqrt{\text{tr}(\Pi)}} P_{\tau}  \left(\mathbb{K} \nabla \phi_p \right) \cdot \bm{v}_f \mathrm{d}{\Gamma_I}\label{ZZZb}\\
	& \qquad \quad \: \,- g \int_{\Gamma_I} \left( \bm{u}_f \cdot \bm{n}_f \right) \psi_p \mathrm{d}x + \int_{\Gamma_I} \frac{\alpha \nu \sqrt{d}}{\sqrt{\text{tr}(\Pi)}} P_{\tau}  \left(\bm{u}_f \right) \cdot \bm{v}_f \mathrm{d}{\Gamma_I},\nonumber\\
	& B(\underline{\bm{v}},p_f) = \mathbb{E}\left[ b(\underline{\bm{v}},p_f) \right] = \int_\Omega  b(\underline{\bm{v}},p_f) \mathrm{d} \omega, \label{ZZZc}\\
	& b(\underline{\bm{v}},p_f) = \int_{D_f}  p_f \nabla \cdot \bm{v}_f \mathrm{d} x, \label{ZZZd}\\
	& F(\underline{\bm{v}}) = \mathbb{E}\left[ f(\underline{\bm{v}}) \right] = \int_\Omega  f(\underline{\bm{v}}) \mathrm{d} \omega, \label{ZZZe}\\
	& f(\underline{\bm{v}}) = g \int_{D_p}  f_p \psi_p \mathrm{d} x + \int_{D_f}  \bm{f}_f \cdot \bm{v}_f \mathrm{d} x + \int_{\Gamma_I} gz \bm{v}_f \cdot \bm{n}_f \mathrm{d}{\Gamma_I}. \label{ZZZf}
\end{align}
\end{subequations}
Here $P_{\tau}  \left(\cdot \right)$ denotes the orthogonal projection onto the local tangential plane of the interface $\Gamma_I$, which is defined as $P_{\tau}  \left(\bm{u} \right) = \bm{u} - \left( \bm{u} \cdot \bm{n}_f\right) \bm{n}_f$.

To ensure the isotropy of the hydraulic conductivity, we assume it to be a diagonal matrix as $\mathbb{K}(\omega,x) = K(\omega,x)\mathbb{I} = \text{diag}\left( K_{11}(\omega,x),\cdots,K_{dd}(\omega,x)\right), \omega \in \Omega, x \in \bar{D}_p = D_p \cup \partial D_p, d = 2,3$, and satisfy the following assumption.

\begin{assumption} [Strong elliptic condition \cite{pao2012nonlinear}] \label{ass3.1}

There exist constants $K_{\min}$ and $K_{\max}$ such that the stochastic hydraulic conductivity satisfies the following uniform ellipticity condition,

\begin{equation}
	0 \, < \, K_{\min} \leq \{K_{ii}( \omega,x)\}_{i=1}^d \leq K_{\max} \, < \, \infty, \quad a.e. \enspace (\omega,x) \enspace in \enspace (\Omega, \bar{D}_p).
	\nonumber
\end{equation}

\end{assumption}

Then we recall the well-posedness of the weak solution for the stochastic coupled problem.

\begin{lemma} [Well-Posedness of the weak solution \cite{yang2022multigrid}] \label{le3.1}
Let Assumption \ref{ass3.1} hold. If the coefficient $\alpha$ in the Beavers–Joseph condition (\ref{ZZc}) is sufficiently small, then the weak formulation (\ref{eq3.6}) of the stochastic Stoke-Darcy interface problem (\ref{eq3.2})-(\ref{eq3.5}) admits a unique and bounded weak solution $\left( \underline{\bm{u}} ,p_f\right) \in \underline{\bm{X}} \times Q_f$ up to an additive constant.
\end{lemma} 

\subsection{Realizations of the random hydraulic conductivity}\label{subsec3.3}

Due to the randomness in hydraulic conductivity, how to approximate the random field $\mathbb{K}(\omega,x)$ properly becomes a key problem in uncertainty quantification for the interface problem. As stated previously, the random hydraulic conductivity is assumed to be in the form of  $\mathbb{K}(\omega,x) = \text{diag}\left( K_{11} \\ (\omega,x),\cdots, K_{dd}(\omega,x)\right), \omega \in \Omega, x \in \bar{D}_p, d = 2,3$. The process of generating the realizations of $K(\omega,x) = K_{11}(\omega,x) $ is displayed in the following, and $K_{ii}( \omega,x),i=2,3$ are obtained in the same manner.

In realistic industrial applications, statistical learning, including linear regressions and Bayesian ridge regressions, is extensively applied in predicting the permeability in porous media regions \cite{roding2020predicting}. Therefore, based on the idea of regression analysis, the random hydraulic conductivity is set to have the following perturbation form that is in line with the practical need:

\begin{equation} \label{eq3.8}
K(\omega,x) = \bar{K}(x) + \widetilde{K}(\omega,x).
\end{equation}

Specifically, the deterministic term $\bar{K}(x)$ is the expectation of the random field $K(\omega,x)$. And without loss of generality, the stochastic term $\widetilde{K}(\omega,x)$ is a finite-dimensional Gaussian random parameter with zero mean. One of the most commonly used examples following the formulation in (\ref{eq3.8}) is the famous Karhunen-Loe\.{v}e (KL) expansion, and Mercer’s theorem guarantees its uniform convergence \cite{gunzburger2014stochastic}. Consequently, this work employs the truncated KL expansion to represent the random hydraulic conductivity $\mathbb{K}(\omega,x)$.

Given the expectation

$$\bar{K}(x) := \mathbb{E}\left[ K(\omega,x)\right],  \quad \text{for each } x \in \bar{D}_p,$$

\noindent and the covariance function 

$$\text{Cov}(x,x^{\prime}) :=\mathbb{E}\left[ \left( K(\omega,x) - \bar{K}(x)\right) \left( K(\omega,x^{\prime}) - \bar{K}(x^{\prime})\right) \right],  \quad \text{for each pair } x,x^{\prime} \in \bar{D}_p,$$

\noindent of the random field $K(\omega,x)$, we consider the following eigenvalue problem

\begin{equation} 
\int_\Omega \text{Cov}(x,x^{\prime}) r(x^{\prime}) \mathrm{d}x^{\prime} = \lambda r(x).
\nonumber
\end{equation}

Owing to the symmetry and positive definiteness of the covariance function, the eigenvalues $\{\lambda_t\}_{t=1}^\infty$ are real and strictly positive. The eigenfunctions $\{r_t(x)\}_{t=1}^\infty$ form a complete orthonormal set in $\mathcal{L}^2(D)$ and satisfy the orthonormality condition as follows

\begin{equation}
\int_\Omega r_t(x) r_{t^{\prime}}(x) \mathrm{d}x = \delta_{t t^{\prime}},
\nonumber
\end{equation}

\noindent where $\delta_{t t^{\prime}}$ is the Kronecker delta. In addition, the eigenvalues are all positive for the positivity of $\text{Cov}(\cdot,\cdot)$. Then, without loss of generality, we may order them in the non-increasing order $\lambda_1 \ge \lambda_2 \ge \cdots \ge 0$. Define

\begin{equation}
Y_t(\omega) = \frac{1}{\sqrt{\lambda_t}} \int_\Omega \left( K(\omega,x) - \bar{K}(x) \right)  r_t(x) \mathrm{d}x.
\nonumber
\end{equation}

\noindent Here $\{Y_t(\omega)\}_{t=1}^\infty$ is a set of mutually uncorrelated real-valued random variables with zero mean and unit variance, i.e., $\mathbb{E}\left[ Y_t(\omega)\right] =0,\mathbb{E}\left[ Y_t(\omega) Y_{t^{\prime}}(\omega)\right] =\delta_{t t^{\prime}}$. 

Therefore, we obtain the following KL expansion \cite{ghanem2003stochastic}

\begin{equation}\label{eq3.9} 
K(\omega,x) = \bar{K}(x) + \displaystyle \sum_{t=1}^{\infty} \sqrt{\lambda_t} r_t(x) Y_t(\omega),
\end{equation}

\noindent where the stochastic term $\widetilde{K}(\omega,x)$ is represented as an infinite-dimensional sum of Gaussian random variables. The availability of (\ref{eq3.9}) arises from the fact that the eigenvalues $\{\lambda_t\}_{t=1}^\infty$ decay as $t$ increases. The decay rate is intrinsically linked to the smoothness of the covariance function $ \text{Cov}(x,x^{\prime})$ and the correlation length \cite{frauenfelder2005finite, todor2005sparse}. Based on a prescribed tolerance $\epsilon$, we may retain only the first $T$ terms such that

\begin{equation}
\frac{\sum\limits_{t=T+1}^{\infty}\lambda_t}{\sum\limits_{t=1}^{\infty}\lambda_t} \le \epsilon, \quad \text{or equivalently, } \quad \frac{\sum\limits_{t=1}^{T}\lambda_t}{\sum\limits_{t=1}^{\infty}\lambda_t} \ge 1-\epsilon.
\nonumber
\end{equation}

Finally, the random field $K(\omega,x)$ is represented by the truncated KL expansion as below

\begin{equation} \label{eq3.10}
\begin{aligned}
	K(\omega,x) \approx K_T(\omega,x) &:= \bar{K}(x) + \widetilde{K}_T(\omega,x) \\
	&=\:\bar{K}(x) + \displaystyle \sum_{t=1}^{T} \sqrt{\lambda_t}\,r_t(x)\, Y_t(\omega),
\end{aligned}
\end{equation}

\noindent where $\{\lambda_t, r_t(x)\}_{t=1}^T$ are the dominant eigenvalues and the corresponding eigen-functions for the covariance kernel $ \text{Cov}(x,x^{\prime})$. Through the truncated KL expansion (\ref{eq3.10}), the random permeability $K(\omega,x)$ is approximated by a perturbation form, in which the deterministic term $\bar{K}(x)$ denotes the mean of $K(\omega,x)$ and the finite dimensional noise $ \widetilde{K}_T(\omega,x)$ depicts the influence of perturbations to the coefficient. 

Furthermore, Assumption \ref{ass3.1} restricts that the range of the random hydraulic conductivity $K(\omega,x)$ should be made finite at both ends of the interval. Therefore, in this work, we assume $\{Y_t(\omega)\}_{t=1}^T$ to denote a sequence of independent identically distributed (i.i.d.) random variables following the truncated standard normal distribution \cite{burkardt2014truncated}, which could preserve the main features of the previous standard normal distribution while avoiding extreme values.

\subsection{Monte Carlo finite element method} \label{subsec3.4}

In this subsection, we focus on the discrete formulation of the stochastic Stokes-Darcy coupled model. The Monte Carlo finite element method (MCFEM) \cite{gunzburger2014stochastic} is utilized in this work to alleviate the curse of dimensionality. Specifically, we approximate the integral $\mathbb{E}[\cdot]$ in  (\ref{eq2.5}) numerically by sample averages of realizations corresponding to the i.i.d. random inputs, which in our case refer to $\{Y_t(\omega)\}_{t=1}^T$ following the truncated standard normal distribution in (\ref{eq3.10}). For spatial discretization concerning $x \in D$, the standard finite element (FE) method is adopted. 

Let $U_h = span\{a_j\}_{j=1}^{N_1} \subset \mathcal{H}^1(D_p)$ be the finite element space for the hydraulic head in the Darcy domain, and $\bm{V}_h = span\{b_j\}_{j=1}^{N_2} \subset \bm{\mathcal{H}}^1(D_f)$, $W_h = span\{c_j\}_{j=1}^{N_3} \subset \mathcal{L}^2(D_f)$ be the finite element space for the velocity and pressure in the Stokes equation. Then the Galerkin formulation is to find $\underline{\bm{u}_h} = (\phi_p^h, \bm{u}_f^h) \in U_h \times \bm{V}_h$ and $p_f^h \in W_h$ such that

\begin{equation} \label{eq3.11}
\left\{
\begin{aligned}
	& A(\underline{\bm{u}_h}, \underline{\bm{v}_h}) - B(\underline{\bm{u}_h},p_f^h) = F(\underline{\bm{v}_h}), \quad \forall \underline{\bm{v}} \in U_{h0} \times \bm{V}_{h0},\\
	& B(\underline{\bm{u}_h},q_f^h) = 0,\qquad\qquad\qquad\; \quad \forall q_f^h \in W_h,
\end{aligned}
\right.
\end{equation}

\noindent where $U_{h0}, \bm{V}_{h0}$ consist of the functions of $U_{h}, \bm{V}_{h}$ with value $0$ on the Dirichlet boundary.

Let $\zeta(x) = \left\lbrace \{a_j\}_{j=1}^{N_1}, \{b_j\}_{j=1}^{N_2}, \{c_j\}_{j=1}^{N_3} \right\rbrace $ be the basis functions, then we have

\begin{equation} 
\phi_p^h (\omega,x) = \displaystyle \sum_{j=1}^{N_1} \phi_j(\omega) a_j(x),  \enspace \bm{u}_f^h(\omega,x)  = \begin{pmatrix} \sum\limits_{j=1}^{N_2} u_{1,j}(\omega) b_j(x) \\ \sum\limits_{j=1}^{N_2} u_{2,j}(\omega) b_j(x)	\end{pmatrix},  \enspace p_f^h (\omega,x) = \displaystyle \sum_{j=1}^{N_3} p_j(\omega) c_j(x).
\nonumber
\end{equation}

Taking the test functions $\underline{\bm{v}_h} = (a_i,0,0)^T\:(i =1,\cdots,N_1)$, $\underline{\bm{v}_h} = (0,b_i,0)^T\:(i =1,\cdots,N_2)$, $\underline{\bm{v}_h} = (0,0,b_i)^T\:(i =1,\cdots,N_2)$ and $q_f^h = c_i\:(i =1,\cdots,N_3)$ respectively, the Galerkin formulation in (\ref{eq3.11}) is expressed as

%\begin{equation}
%\left\{\enspace
\begin{align*}
& \int_{D_p} K_T(\omega) \nabla \left(\displaystyle \sum_{j=1}^{N_1} \phi_j(\omega) a_j\right) \cdot \nabla a_i \mathrm{d}D_p - \int_{\Gamma_I} \left[\left(\displaystyle \sum_{j=1}^{N_2} u_{1,j}(\omega) b_j\right)n_1 + \left(\displaystyle \sum_{j=1}^{N_2} u_{2,j}(\omega) b_j\right)n_2 \right] b_i\mathrm{d}S  =  \int_{D_p} f_p a_i \mathrm{d}D_p,\\
& \int_{D_f} 2 \nu \left(\displaystyle \sum_{j=1}^{N_2} u_{1,j}(\omega) \frac{\partial b_j}{\partial x}\right) \frac{\partial b_i}{\partial x} \mathrm{d}D_f + \int_{D_f} \nu \left(\displaystyle \sum_{j=1}^{N_2} u_{1,j}(\omega)  \frac{\partial b_j}{\partial y}\right) \frac{\partial b_i}{\partial y} \mathrm{d}D_f + \int_{D_f} \nu \left(\displaystyle \sum_{j=1}^{N_2} u_{2,j}(\omega) \frac{\partial b_j}{\partial x}\right) \frac{\partial b_i}{\partial y} \mathrm{d}D_f \\ 
& \quad - \int_{D_f} \left(\displaystyle \sum_{j=1}^{N_3} p_{j}(\omega) c_j\right) \frac{\partial b_i}{\partial x} \mathrm{d}D_f + \int_{\Gamma_I} \delta \tau_1^2 \left(\displaystyle \sum_{j=1}^{N_2} u_{1,j}(\omega) b_j\right) b_i \mathrm{d}S + \int_{\Gamma_I} \delta \tau_1 \tau_2 \left(\displaystyle \sum_{j=1}^{N_2} u_{2,j}(\omega) b_j\right) b_i \mathrm{d}S\\
& \qquad \qquad \qquad \qquad \qquad \qquad\qquad + g \int_{\Gamma_I} n_1 \left(\displaystyle \sum_{j=1}^{N_1} \phi_{j}(\omega) a_j\right) b_i \mathrm{d}S = \int_{D_f} f_{s,1} b_i \mathrm{d}D_f + \int_{\Gamma_I} gzn_{1} b_i \mathrm{d}S,\\
& \int_{D_f} 2 \nu \left(\displaystyle \sum_{j=1}^{N_2} u_{2,j}(\omega) \frac{\partial b_j}{\partial y}\right) \frac{\partial b_i}{\partial y} \mathrm{d}D_f + \int_{D_f} \nu \left(\displaystyle \sum_{j=1}^{N_2} u_{1,j}(\omega)  \frac{\partial b_j}{\partial y}\right) \frac{\partial b_i}{\partial x} \mathrm{d}D_f + \int_{D_f} \nu \left(\displaystyle \sum_{j=1}^{N_2} u_{2,j}(\omega) \frac{\partial b_j}{\partial x}\right) \frac{\partial b_i}{\partial x} \mathrm{d}D_f \\
& \qquad - \int_{D_f} \left(\displaystyle \sum_{j=1}^{N_3} p_{j}(\omega) c_j\right) \frac{\partial b_i}{\partial y} \mathrm{d}D_f + \int_{\Gamma_I} \delta \tau_2 \tau_1 \left(\displaystyle \sum_{j=1}^{N_2} u_{1,j}(\omega) b_j\right) b_i \mathrm{d}S + \int_{\Gamma_I} \delta \tau_2^2 \left(\displaystyle \sum_{j=1}^{N_2} u_{2,j}(\omega) b_j\right) b_i \mathrm{d}S\\
& \qquad \qquad \qquad \qquad \qquad \qquad\qquad + g \int_{\Gamma_I} n_2 \left(\displaystyle \sum_{j=1}^{N_1} \phi_{j}(\omega) a_j\right) b_i \mathrm{d}S = \int_{D_f} f_{s,2} b_i \mathrm{d}D_f + \int_{\Gamma_I} gzn_{2} b_i \mathrm{d}S,\\
& \int_{D_f} \left(\displaystyle \sum_{j=1}^{N_2} u_{1,j}(\omega) \frac{\partial b_j}{\partial x}\right) c_i \mathrm{d}D_f + \int_{D_f} \left(\displaystyle \sum_{j=1}^{N_2} u_{2,j}(\omega) \frac{\partial b_j}{\partial y}\right) c_i \mathrm{d}D_f = 0.
\end{align*}
%\right.
%\end{equation}

Given the MC realizations $K_T^m \: (m=1,\cdots,M)$ of the random conductivity $K_T(\omega)$ in the truncated KL expansion in (\ref{eq3.10}), (\ref{eq3.11}) is rewritten into the following $M$ linear systems of algebraic equations

\begin{equation} \label{eq3.12}
\mathbf{A}_m \bm{x}_h^m = \bm{b}, \quad m = 1,2,\cdots,M,
\end{equation}

\noindent and the QoI, that is the statistical characterization of the system response, is now estimated by

\begin{equation} \label{eq3.13}
\begin{aligned}
	\hat{\bm{x}}_{h,M}(x) := &\frac{1}{M} \displaystyle \sum_{m=1}^M \displaystyle \sum_{j=1}^N \bm{x}_{h}^m \:\zeta_j(x)\\
	= &\frac{1}{M} \displaystyle \sum_{m=1}^M \left[ \displaystyle \sum_{j=1}^{N_1} \phi_{h}^m a_j(x) + \displaystyle \sum_{j=1}^{N_2} \bm{u}_{1,h}^m b_j(x) + \displaystyle \sum_{j=1}^{N_2} \bm{u}_{2,h}^m b_j(x) + \displaystyle \sum_{j=1}^{N_3} \bm{p}_{h}^m c_j(x) \right],
\end{aligned}
\end{equation}

\noindent where $\{\zeta_j(x)\}_{j=1}^N, N := N_1+2N_2+N_3$ denotes the overall FE basis functions.

Specifically, the stiffness matrices $\mathbf{A}_m $,  the MCFEM sample solutions $\bm{x}_h^m $ and the load vector $\bm{b}$ of the linear systems  in (\ref{eq3.12}) are respectively given by

\begin{equation} \label{eq3.14}
\begin{aligned}
	&\mathbf{A}_m = 
	\begin{pmatrix}
		\mathbf{P}_m & -\mathbf{I1} & -\mathbf{I2} & \mathbf{O}_{N_1\times N_3}\\
		\mathbf{I3} + \mathbf{I9}_m + \mathbf{I11}_m & 2\mathbf{F1} + \mathbf{F2} + \mathbf{I5} & \mathbf{F3} + \mathbf{I7} & \mathbf{F5}\\
		\mathbf{I4} + \mathbf{I10}_m + \mathbf{I12}_m & \mathbf{F4} + \mathbf{I8} & \mathbf{F1} + 2\mathbf{F2} + \mathbf{I6} &  \mathbf{F6}\\
		\mathbf{O}_{N_3 \times N_1} & \mathbf{F5}^T  & \mathbf{F6}^T & \mathbf{O}_{N_3 \times N_3}
	\end{pmatrix},\\
	&\bm{x}_h^m= \left(\phi_h^m, \bm{u}_{1,h}^m,\bm{u}_{2,h}^m, \bm{p}_h^m\right) ^T, \bm{b} = \left(\bm{b1},\bm{b2},\bm{b3},\bm{O}_{1\times N_3}\right) ^T, \phi_h^m = \left( \phi_1^m, \cdots, \phi_{N_1}^m\right), \\
	& \bm{u}_{1,h}^m = \left( u_{1,1}^m, \cdots, u_{1,N_2}^m \right), \enspace\bm{u}_{2,h}^m = \left( u_{2,1}^m, \cdots, u_{2,N_2}^m \right), \enspace\bm{p}_h^m=\left( p_1^m,\cdots, p_{N_3}^m\right),
\end{aligned}
\end{equation}

\noindent and the sub-matrices and sub-vectors are defined as below

\begin{equation}\label{eq3.15}
\begin{aligned}
	& \mathbf{P}_m = \left[ \int_{D_p} K_T^m \nabla a_j \nabla a_i \mathrm{d} D_p\right] _{i,j=1}^{N_1}, \qquad \quad \mathbf{F1} = \left[ \int_{D_f} \nu \frac{\partial b_j}{\partial x} \frac{\partial b_i}{\partial x}\mathrm{d} D_f\right] _{i,j=1}^{N_2},\\ 
	& \mathbf{F2} = \left[ \int_{D_f} \nu \frac{\partial b_j}{\partial y} \frac{\partial b_i}{\partial y}\mathrm{d} D_f\right] _{i,j=1}^{N_2}, \qquad \quad \enspace \:\mathbf{F3} = \left[ \int_{D_f} \nu \frac{\partial b_j}{\partial x} \frac{\partial b_i}{\partial y}\mathrm{d} D_f\right] _{i,j=1}^{N_2},\\ 
	& \mathbf{F4} = \left[ \int_{D_f} \nu \frac{\partial b_j}{\partial y} \frac{\partial b_i}{\partial x}\mathrm{d} D_f\right] _{i,j=1}^{N_2}, \qquad \quad \enspace \: \mathbf{F5} = \left[- \int_{D_f} c_j \frac{\partial b_i}{\partial x}\mathrm{d} D_f\right] _{i,j=1}^{N_2,N_3},\\ 
	& \mathbf{F6} = \left[- \int_{D_f} c_j \frac{\partial b_i}{\partial y}\mathrm{d} D_f \right] _{i,j=1}^{N_2,N_3},\qquad \quad \enspace \enspace \mathbf{I1} =  \left[ \int_{\Gamma_I} n_1 b_j a_i \mathrm{d}S \right] _{i,j=1}^{N_1,N_2},\\ 
	& \mathbf{I2} =  \left[ \int_{\Gamma_I} n_2 b_j a_i \mathrm{d}S \right] _{i,j=1}^{N_1,N_2}, \qquad \quad \quad \quad \enspace  \mathbf{I3} =  \left[ \int_{\Gamma_I} g n_1 a_j b_i \mathrm{d}S \right] _{i,j=1}^{N_2,N_1},\\ 
	& \mathbf{I4} =  \left[ \int_{\Gamma_I} g n_2 a_j b_i \mathrm{d}S \right] _{i,j=1}^{N_2,N_1}, \qquad \quad \quad \quad  \mathbf{I5} =  \left[ \int_{\Gamma_I} \delta \tau_1^2 b_j b_i \mathrm{d}S \right] _{i,j=1}^{N_2},\\ 
	& \mathbf{I6} =  \left[ \int_{\Gamma_I} \delta \tau_2^2 b_j b_i \mathrm{d}S \right] _{i,j=1}^{N_2}, \qquad \quad \quad \quad \,\; \mathbf{I7} =  \left[ \int_{\Gamma_I} \delta \tau_1 \tau_2 b_j b_i \mathrm{d}S \right] _{i,j=1}^{N_2},\\ 
	& \mathbf{I8} =  \left[ \int_{\Gamma_I} \delta \tau_2 \tau_1 b_j b_i \mathrm{d}S \right] _{i,j=1}^{N_2}, \qquad \quad \quad \: \;\mathbf{I9}_m =  \left[ \int_{\Gamma_I} \delta \tau_1^2 K_T^m \frac{\partial a_j}{\partial x} b_i \mathrm{d}S \right] _{i,j=1}^{N_2,N_1},\\ 
	& \mathbf{I10}_m =  \left[ \int_{\Gamma_I} \delta \tau_2^2 K_T^m \frac{\partial a_j}{\partial x} b_i \mathrm{d}S \right] _{i,j=1}^{N_2,N_1},\enspace \quad \:\: \mathbf{I11}_m =  \left[ \int_{\Gamma_I} \delta \tau_1 \tau_2 K_T^m \frac{\partial a_j}{\partial x} b_i \mathrm{d}S \right] _{i,j=1}^{N_2,N_1},\\ 
	& \mathbf{I12}_m =  \left[ \int_{\Gamma_I} \delta \tau_2 \tau_1 K_T^m \frac{\partial a_j}{\partial x} b_i \mathrm{d}S \right] _{i,j=1}^{N_2,N_1}, \enspace\:\:\,\bm{b1} = \left[ \int_{D_p} f_p a_i \mathrm{d}D_p \right] _{i=1}^{N_1},\\ 
	& \bm{b2} \enspace = \enspace \left[ \int_{D_f} \:f_{f1} b_i \: \mathrm{d}D_f \right] _{i=1}^{N_2} \enspace + \enspace \left[ \int_{\Gamma_I}\:g z n_1 b_i \:\mathrm{d}S \right] _{i=1}^{N_2},\\ 
	& \bm{b3} \enspace = \enspace \left[ \int_{D_f} \:f_{f2} b_i \: \mathrm{d}D_f \right] _{i=1}^{N_2} \enspace + \enspace \left[ \int_{\Gamma_I}\:g z n_2 b_i \:\mathrm{d}S \right] _{i=1}^{N_2}.
\end{aligned}
\end{equation}

According to (\ref{eq3.14}) and (\ref{eq3.15}), the coupled Stokes-Darcy model has high degrees of freedom and we need to face a sequence of large-scale stiffness matrices $\{\mathbf{A}_m\}_{m=1}^M$ with the dimension of $N = N_1+2N_2+N_3$. In addition, the MC convergence $\mathcal{O}(\sqrt{\frac{1}{M}})$ is related to the sample size $M$, i.e., it requires a large number of samples to obtain a reliable QoI. Therefore, the dimensions in both the physical and stochastic space become very large, which results in heavy computational and storage burdens for the standard MCFEM method. Hence, it is in a great need to develop more efficient solvers for this complex system.

\section{Low-rank solver for the Stokes-Darcy coupled model with random hydraulic conductivity}\label{sec4}

In order to reduce temporal and spatial complexities, it is natural to explore dimensionality reduction strategies for dealing with high-dimensional stiffness matrices. This section presents a novel generalized low-rank matrix approximation technique, employs it to establish an efficient low-rank solver for the stochastic Stokes-Darcy coupled system, and carries out the error analysis.

\subsection{Generalized low-rank approximation of matrices} \label{subsec4.1}

The nature of low-rank matrix approximation is dimensionality reduction. The goal is to derive a more compact data representation that preserves as much intrinsic information as possible. Therefore, solving a given problem on a lower-rank alternative yields similar numerical accuracy to solving it on the original matrices, which effectively reduces our computational and storage burden. 

Mathematically, the optimal rank-$k$ approximation ($k \le N$) of a matrix $\mathbf{B} \in \mathbb{R}^{N \times N}$ under the Frobenius norm can be formulated as the following rank-constrained minimization problem: compute the optimal matrix $\mathbf{B}^* \in \mathbb{R}^{N \times N}$, which solves

\begin{equation} \label{eq4.1}
\mathop{\min}\limits_{\text{rank}(\mathbf{B}^*) \; = \; k} \enspace \Vert \mathbf{B} - \mathbf{B}^* \Vert_F,
\end{equation}

\noindent where $ \Vert \mathbf{B} - \mathbf{B}^* \Vert_F$ is called as the reconstruction error of the low-rank matrix approximation. 

The famous Eckart-Young-Mirsky theorem provides the theoretical guide for solving the optimization problem above, demonstrating that the optimal solution of (\ref{eq4.1}) can be given explicitly by the truncated singular value decomposition of the original matrix. 

\begin{theorem}[\cite{eckart1936approximation}]\label{th4.1} 

Let $\mathbf{B} = \mathbf{U} \Sigma \mathbf{V}^T \in \mathbb{R}^{N \times N}$ denote the singular value decomposition of $\mathbf{B}$. Then the optimal rank-$k$ approximation of $\mathbf{B}$ under the Frobenius norm is given by the truncated SVD as follows:
$$\mathbf{B}^* \:= \:\mathbf{U}_k \Sigma_k \mathbf{V}_k^T,$$

\noindent where $\mathbf{U}_k, \mathbf{V}_k \in \mathbb{R}^{N \times k}$ contain the first $k$ left and right singular vectors of $\mathbf{B}$, and $\Sigma_k \in \mathbb{R}^{k \times k}$ is a diagonal matrix of the corresponding largest $k$ singular values. Moreover, $\mathbf{B}^*$ is the unique minimizer if and only if the singular values $\{\sigma_i\}_{i=1}^N$ of $\mathbf{B}$ follow $\sigma_{k+1} \neq \sigma_k$.

\end{theorem}

Despite its appealing precision in data reconstruction, SVD is only well-suited for dealing with a single matrix, and when applied to a collection of high-dimensional matrices, it has to face practical limits arising from expensive computational costs and heavy memory demands. As a result, a "generalized" idea has been introduced to low-rank matrix approximation \cite{ye2004generalized}, which aims to find a generalized tri-decomposition form for a sequence of matrices $\{\mathbf{B}_m\}_{m=1}^M$ as below

\begin{equation} \label{eq4.2}
\mathop{\min}\limits_{\mathbf{U}^T \mathbf{U} = \mathbf{I}_k , \mathbf{V}^T \mathbf{V} = \mathbf{I}_k} \quad \displaystyle \sum_{m=1}^M \Vert \mathbf{B}_m - \mathbf{U} \mathbf{S}_m \mathbf{V}^T \Vert_F^2,
\end{equation}

\noindent where $\mathbf{U}, \mathbf{V} \in \mathbb{R}^{N \times k}$ are the fixed two-sided matrices with orthonormal columns that are shared by the set $\{\mathbf{B}_m\}_{m=1}^M$. 

The generalized low-rank approximation of matrices could effectively alleviate the high SVD calculation and memory burden. Therefore, we would like to utilize GLRAM techniques to solve the stochastic Stokes-Darcy system and further reduce the computational and space complexities from the linear systems in (\ref{eq3.12}). However, there are still several difficulties in combining the current GLRAM technology with solving SPDEs. On one hand, it has relatively expensive computational complexity due to its iterative structure \cite{ye2004generalized}. On the other hand, the standard GLRAM is generally employed in the fields of computer vision and signal processing, and it is not straightforward to achieve the matrix reconstruction accuracy pursued in numerical PDE settings. 

Therefore, this work introduces a novel GLRAM strategy that achieves a more efficient balance among dimensionality reduction efficacy, computational efficiency, and numerical precision. The proposed technique demonstrates strong performance in solving the stochastic Stokes-Darcy interface model. In contrast to the conventional SVD and GLRAM frameworks, our approach utilizes a different low-rank matrix representation, which allows better data compression capability and higher fidelity in matrix reconstruction.

Given a pre-specified data compression ratio $\theta \in [0,1]$, the reduced dimension $k$ is determined as below,

\begin{equation}
k \: = \: \lceil \theta N \rceil \: \in \: [0,N],
\nonumber
\end{equation}

\noindent and then the optimal generalized rank-$k$ matrix approximation of $\{\mathbf{B}_m\}_{m=1}^M \in \mathbb{R}^{ N\times N}$ is constructed by finding the matrices $\mathbf{U}, \{\mathbf{V}_m\}_{m=1}^M \in \mathbb{R}^{N \times k}$ such that

\begin{equation} \label{eq4.3}
\mathop{\min}\limits_{ \substack{\mathbf{U},\{\mathbf{V}_m\}_{m=1}^M \in \mathbb{R}^{N \times k} \\ \mathbf{U}^T \mathbf{U} = \mathbf{I}_k}} \quad  \displaystyle \sum_{m=1}^M \Vert \mathbf{B}_m - \mathbf{U} \mathbf{V}_m^T \Vert_F^2.
\end{equation}

To construct our method to solve the minimization problem in (\ref{eq4.3}), we first look at the relationship between the semi-orthogonal matrices $\mathbf{U}$ and $\mathbf{V}_m$.

\begin{theorem}\label{th4.2}

Let $\mathbf{U}$ and $\{\mathbf{V}_m\}_{m=1}^M$ be the optimal solution to the minimization problem in (\ref{eq4.3}), then     
$$\mathbf{V}_m \:= \:\mathbf{B}_m^T \mathbf{U}, \quad m = 1,2,\cdots,M.$$

\end{theorem}

\begin{proof}

By the symmetry of the Frobenius inner product, we have

\begin{equation} 
	\begin{split}
		\displaystyle \sum_{m=1}^M \: \Vert \mathbf{B}_m - \mathbf{U} \mathbf{V}_m^T \Vert_F^2 \: &= \: \displaystyle \sum_{m=1}^M \: \text{tr}\left( \left( \mathbf{B}_m - \mathbf{U} \mathbf{V}_m^T\right)  \left( \mathbf{B}_m - \mathbf{U} \mathbf{V}_m^T\right) ^T\right) \\
		&= \: \displaystyle \sum_{m=1}^M \: \text{tr}\left( \mathbf{B}_m \mathbf{B}_m^T\right) + \displaystyle \sum_{m=1}^M \: \text{tr}\left( \mathbf{V}_m \mathbf{V}_m^T\right) - 2 \displaystyle \sum_{m=1}^M \: \text{tr}\left( \mathbf{B}_m^T \mathbf{U} \mathbf{V}_m^T\right), 
		\nonumber
	\end{split}
\end{equation}

\noindent in which the second term $\text{tr}\left( \mathbf{V}_m \mathbf{V}_m^T\right) = \text{tr}\left( \left(\mathbf{U}\mathbf{V}_m^T\right)\left(\mathbf{U}\mathbf{V}_m^T\right)^T\right)$ holds, since $\mathbf{U}$ has orthonormal columns. 

It is easy to observe that the first term $\displaystyle \sum_{m=1}^M \: \text{tr}\left( \mathbf{B}_m \mathbf{B}_m^T\right)$ is a constant, and then (\ref{eq4.3}) is equivalent to the following optimization problem

\begin{equation} 
	\mathop{\min}\limits_{\mathbf{U}, \mathbf{V}_m} \: \displaystyle \sum_{m=1}^M \: \text{tr}\left( \mathbf{V}_m \mathbf{V}_m^T\right)  \: - \: 2 \displaystyle \sum_{m=1}^M \: \text{tr}\left( \mathbf{B}_m^T \mathbf{U} \mathbf{V}_m^T\right). 
	\nonumber
\end{equation}

\noindent According to standard properties of the matrix trace, the minimization problem above reaches the optimal value only if $\mathbf{V}_m = \mathbf{B}_m^T \mathbf{U}$ and $\mathbf{V}_m^T = \mathbf{U}^T \mathbf{B}_m$ holds for all $m = 1,2,\cdots, M$.

\end{proof}

Then the following lemma is introduced to establish our main result in the follow-up theorem.

\begin{lemma}\label{le4.1}
Let $\mathbf{R} \in \mathbb{R}^{N \times N}$ be a symmetric matrix and $\mathbf{I}_k$ denote the $k$-dimensional identity matrix $(k \leq N)$. Then the following optimization problem has an optimal solution $\mathbf{U} \in \mathbb{R}^{N \times k}$, whose columns are equal to the first $k$ dominant eigenvectors of $\mathbf{R}$,

\begin{equation}
	\begin{split}
		\max & \quad \text{tr}\left( \mathbf{U}^T \mathbf{R} \mathbf{U}\right) ,\\
		s.t. & \quad \mathbf{U}^T \mathbf{U} = \mathbf{I}_k.
		\nonumber
	\end{split}
\end{equation}

\end{lemma}

\begin{proof}
Since $\mathbf{R}$ is a real-valued symmetric matrix, it has the eigen-decomposition $\mathbf{R} = \mathbf{Q} \Lambda \mathbf{Q}^T$, where the orthogonal matrix $\mathbf{Q}$ is composed of the eigenvectors of $\mathbf{R}$, and the diagonal matrix $\Lambda = \text{diag}(\lambda_1, ..., \lambda_N)$, $\lambda_1 \geq \cdots \geq \lambda_N$ includes the eigenvalues of $\mathbf{N}$ in the descending order. Both $\mathbf{Q}$ and $\mathbf{U}$ have orthogonal columns, and thus $\mathbf{P} := \mathbf{Q}^T \mathbf{U}$ satisfies $\mathbf{P}^T \mathbf{P} = \mathbf{I}_k$. Then, the objective function can be rewritten as follows,

\begin{equation}
	\begin{aligned}
		\text{tr}\left( \mathbf{U}^T \mathbf{R} \mathbf{U}\right)  \: &= \: \text{tr} \left( \mathbf{U}^T \mathbf{Q} \Lambda \mathbf{Q}^T \mathbf{U}\right) \\
		& = \: \text{tr} \left( \mathbf{P}^T \Lambda \mathbf{P}\right) \: = \: \text{tr} \left( \Lambda \mathbf{P} \mathbf{P}^T\right).
		\nonumber
	\end{aligned}
\end{equation}

\noindent According to the Schur-Horn theorem \cite{horn1954doubly}, we have

\begin{equation}
	\begin{aligned}
		\text{tr}\left( \mathbf{U}^T \mathbf{R} \mathbf{U}\right)  \: &= \: \text{tr} \left( \Lambda \mathbf{P} \mathbf{P}^T\right)\\
		&= \: \displaystyle \sum_{i=1}^N\sum_{j=1}^k \lambda_i \; p_{ij}^2 \: \leq \: \displaystyle \sum_{i=1}^N \lambda_i,
		\nonumber
	\end{aligned}
\end{equation}

\noindent where $\mathbf{P} = \left[p_{ij}\right]_{i,j=1}^{N,k}$. The equality is achievable only if the nonzero entries of $\mathbf{P}$ are $p_{ii}=1,\:i=1,\cdots,k$. Therefore, $\mathbf{Q}_k^T \mathbf{U} = \mathbf{I}_k$, and $\mathbf{U} = \mathbf{Q}_k$ is composed of the first $k$ dominant eigenvectors of $\mathbf{R}$.

\end{proof}

Since Theorem \ref{th4.2} guarantees that $\mathbf{V}_m$ is solely dependent on $\mathbf{U}$ and $\mathbf{B}_m$, we can finally provide the following non-iterative approach for the computation of the semi-orthogonal matrix $\mathbf{U}$ stated in Theorem \ref{th4.3}. 

\begin{theorem}\label{th4.3}

Let $\mathbf{U}$ be the optimal solution to the minimization problem in (\ref{eq4.3}). Then $\mathbf{U}$ consists of the $k$ eigenvectors of the matrix

\begin{equation} \label{eq4.4}
	\mathbf{N} = \displaystyle \sum_{m=1}^M \mathbf{B}_m \mathbf{B}_m^T,
\end{equation}

\noindent corresponding to the first $k$ dominant eigenvalues.
\end{theorem}

\begin{proof}

Based on Theorem \ref{th4.2}, the minimization problem in (\ref{eq4.3}) can be rewritten as follows

\begin{equation}
	\mathop{\max}\limits_{\mathbf{U}} \: \displaystyle \sum_{m=1}^M \:\text{tr}\left(  \mathbf{U} \mathbf{U}^T \mathbf{B}_m \mathbf{B}_m^T\right). 
	\nonumber
\end{equation}

\noindent Then, according to properties of the matrix trace and Lemma \ref{le4.1}, we have

\begin{equation}
	\begin{split}
		\displaystyle \sum_{m=1}^M \:\text{tr}\left(  \mathbf{U} \mathbf{U}^T \mathbf{B}_m \mathbf{B}_m^T\right)  \: 
		&= \: \displaystyle \sum_{m=1}^M \:\text{tr}\left(  \mathbf{U}^T \mathbf{B}_m \mathbf{B}_m^T \mathbf{U}\right)  \\
		&= \: \text{tr}\left(  \mathbf{U}^T \left( \displaystyle \sum_{m=1}^M \mathbf{B}_m \mathbf{B}_m^T\right)  \mathbf{U}\right)  \\
		&= \: \text{tr}\left(  \mathbf{U}^T \mathbf{N} \mathbf{U}\right) .
		\nonumber
	\end{split}
\end{equation}

\noindent Since $\mathbf{N}$ is a symmetric matrix, the optimizer of the maximization system should satisfy that $\mathbf{U}$ is comprised of the first $k$ eigenvectors of $\mathbf{N}$, which corresponds to the largest $k$ eigenvalues of the matrix.

\end{proof}

The process to obtain the optimal solution $\mathbf{U}$ and $\{\mathbf{V}_m\}_{m=1}^M$ of (\ref{eq4.3}) is summarized in Algorithm \ref{alg1}.

\begin{algorithm}[htbp]
\caption{A novel generalized low-rank approximation method of matrices in the bi-decomposition form} \label{alg1}
\begin{algorithmic}[1]
	
	\Require
	Matrices $\{\mathbf{B}_m\}_{m=1}^M$, and the data compression ratio $\theta$
	\Ensure
	Matrices $\mathbf{U}$ and $\{\mathbf{V}_m\}_{m=1}^M$
	
	\State Determine the reduced rank $k = \lceil \theta N \rceil$. 
	
	\State Construct the matrix $\mathbf{N}$ in (\ref{eq4.4}).
	
	\State Form the matrix $\mathbf{U}$ by taking the $k$ eigenvectors of $\mathbf{N}$ corresponding to the largest $k$ eigenvalues.
	
	\For{$m = 1,\cdots,M$}    
	
	\State Obtain the collection of matrices $\mathbf{V}_m \: = \: \mathbf{B}_m^T \; \mathbf{U}, \:m = 1,\cdots,M$.
	
	\EndFor
	
	\State \Return the optimal solution $\mathbf{U}$ and $\{\mathbf{V}\}_{m=1}^M$ of (\ref{eq4.3}).
	
\end{algorithmic}
\end{algorithm}

We also compare the space and computational complexities of the standard SVD, GLRAM and Algorithm \ref{alg1} in Table \ref{tab4.1}, and the outputs illustrate that our novel generalized low-rank matrix approximation method can remarkably save the CPU runtime and storage.

\begin{table}[!ht] 
\centering 
\caption{Comparison of the standard SVD, GLRAM, and Algorithm \ref{alg1}, where $M, N, k, I$ denote the MC sample size, the dimensions of the original and compressed matrices, and the iteration counts in standard GLRAM respectively.}
\begin{tabular}{ccc} \hline 
	Method & Space & Time  \\ \hline
	Traditional SVD & $MN^2$  & $\mathcal{O}(2MN^3)$\\
	Traditional GLRAM & $2Nk + Mk^2$ & $\mathcal{O}(MN^2(2I(N+2k)))$ \\
	Algorithm \ref{alg1} & $Nk + MNk$ & $\mathcal{O}(MN^2(N+k))$\\ \hline
	%Ratio & $\mathcal{O}(\theta)$ & $\mathcal{O}(\theta^2)$ \\ \hline
\end{tabular}
\label{tab4.1}
\end{table}

\subsection{Construct the low-rank solver for the stochastic Stokes-Darcy model} \label{subsec4.2}

The space requirement of storing the stiffness matrices $\{\mathbf{A}_m\}_{m=1}^M$ could be significantly scaled down via our generalized low-rank approximation techniques. Now we consider how to reduce the total computational complexity, which is determined by both the MC sample size and the computational complexity of each sample model. A sufficiently large sample size is necessary to guarantee the precision of the MC method, and the computational cost of individual samples mainly lies in computing the inversion of matrices $\{\mathbf{A}_m\}_{m=1}^M$ in (\ref{eq3.12}). These stiffness matrices arising from FE discretization are generally large and sparse, and direct inversion of $\{\mathbf{A}_m\}_{m=1}^M$ in (\ref{eq3.12}) leads to substantial computational complexity. To address the problem, we first adopt the following pre-processing measure to deal with the stiffness matrices $\{\mathbf{A}_m\}_{m=1}^M$.

According to the truncated KL expansion in (\ref{eq3.10}), the random field $K_T(\omega,x)$ is decomposed into two parts, the deterministic function $\bar{K}(x)$ and the stochastic function $ \widetilde{K}_T(\omega,x)$. Similar to the random hydraulic conductivity, we also adopt the following matrix splitting to the stiffness matrices $\{\mathbf{A}_m\}_{m=1}^M$ of linear systems in (\ref{eq3.12}):

\begin{equation} \label{eq4.5}
\left( \bar{\mathbf{A}} + \widetilde{\mathbf{A}}_m\right)   \bm{x}_h^m =   \bm{b}, \quad m = 1,2,\cdots,M.
\end{equation}

\begin{remark} \label{rmk4.1}

Similarly to (\ref{eq3.8}), the matrix splitting in (\ref{eq4.5}) can also be considered as a perturbation form, in which $\bar{\mathbf{A}}$ illustrates the similarities among $\{\mathbf{A}_m\}_{m=1}^M$ and $\widetilde{\mathbf{A}}_m$ denotes the perturbations arising from the Gaussian field $K_T(\omega,x)$. The procedure of subtracting the statistical mean from each sample is often referred to as data normalization, which is a useful pre-processing method and has been widely applied in signal processing and deep learning \cite{hemanth2017deep}. This technique effectively eliminates redundant commonalities across these MC realizations and makes it easier to observe individual uniqueness, thereby improving the robustness and generalization performance of numerical algorithms.

\end{remark} 

Therefore, $\bar{\mathbf{A}}$ is a deterministic matrix stemmed from $\bar{K}(x)$ and has the form of

\begin{equation} \label{eq4.6}
\bar{\mathbf{A}} = 
\begin{pmatrix}
	\bar{\mathbf{P}} & -\mathbf{I1} & -\mathbf{I2} & \mathbf{O}_{N_1\times N_3}\\
	\mathbf{I3} + \mathbf{\bar{I}9} + \mathbf{\bar{I}11} & 2\mathbf{F1} + \mathbf{F2} + \mathbf{I5} & \mathbf{F3} + \mathbf{I7} & \mathbf{F5}\\
	\mathbf{I4} + \mathbf{\bar{I}10} + \mathbf{\bar{I}12} & \mathbf{F4} + \mathbf{I8} & \mathbf{F1} + 2\mathbf{F2} + \mathbf{I6} &  \mathbf{F6}\\
	\mathbf{O}_{N_3 \times N_1} & \mathbf{F5}^T  & \mathbf{F6}^T & \mathbf{O}_{N_3 \times N_3}
\end{pmatrix},
\end{equation}

\noindent and $\widetilde{\mathbf{A}}_m \: (m = 1,\cdots,M)$ results from the MC realizations of $\widetilde{K}_T(\omega,x)$, which is defined as

\begin{equation} \label{eq4.7}
\widetilde{\mathbf{A}}_m = 
\begin{pmatrix}
	\begin{matrix}
		\widetilde{\mathbf{P}}_m \\
		\mathbf{\widetilde{I}9}_m + \mathbf{\widetilde{I}11}_m  \\
		\mathbf{\widetilde{I}10}_m +\mathbf{\widetilde{I}12}_m
	\end{matrix} & \text{\large$ \mathbf{O}_{(N_1 + 2N_2) \times (2N_2+N_3)}$}\\
	\mathbf{O}_{N_3 \times N_1} & \mathbf{O}_{N_3 \times (2N_2+N_3)}
\end{pmatrix}, \: \: m = 1,\cdots,M.
\end{equation}

Note that $\widetilde{\mathbf{A}}_m$ is not required to be of full rank, and its rank deficiency will have no effect on the subsequent construction of its generalized low-rank representation. According to (\ref{eq3.15}), the sub-matrices in (\ref{eq4.6}) and (\ref{eq4.7}) are given by

\begin{align*}
& \bar{\mathbf{P}}= \left[ \int_{D_p}\bar{K} \nabla a_j \nabla a_i \mathrm{d} D_p\right] _{i,j=1}^{N_1}, \qquad \quad\: \widetilde{\mathbf{P}}_m = \left[ \int_{D_p}\widetilde{K}_T^m \nabla a_j \nabla a_i \mathrm{d} D_p\right] _{i,j=1}^{N_1},\\
& \mathbf{\bar{I}9} =  \left[ \int_{\Gamma_I} \delta \tau_1^2 \bar{K} \frac{\partial a_j}{\partial x} b_i \mathrm{d}S \right] _{i,j=1}^{N_2,N_1}, \qquad \quad\: \mathbf{\widetilde{I}9}_m =  \left[ \int_{\Gamma_I} \delta \tau_1^2 \widetilde{K}_T^m \frac{\partial a_j}{\partial x} b_i \mathrm{d}S \right] _{i,j=1}^{N_2,N_1},\\
& \mathbf{\bar{I}10} =  \left[ \int_{\Gamma_I} \delta \tau_2^2 \bar{K} \frac{\partial a_j}{\partial x} b_i \mathrm{d}S \right] _{i,j=1}^{N_2,N_1}, \qquad \enspace\: \mathbf{\widetilde{I}10}_m =  \left[ \int_{\Gamma_I} \delta \tau_2^2 \widetilde{K}_T^m \frac{\partial a_j}{\partial x} b_i \mathrm{d}S \right] _{i,j=1}^{N_2,N_1},\\
& \mathbf{\bar{I}11} =  \left[ \int_{\Gamma_I} \delta \tau_1 \tau_2 \bar{K} \frac{\partial a_j}{\partial x} b_i \mathrm{d}S \right] _{i,j=1}^{N_2,N_1}, \qquad \mathbf{\widetilde{I}11}_m =  \left[ \int_{\Gamma_I} \delta \tau_1 \tau_2 \widetilde{K}_T^m \frac{\partial a_j}{\partial x} b_i \mathrm{d}S \right] _{i,j=1}^{N_2,N_1},\\
& \mathbf{\bar{I}12} =  \left[ \int_{\Gamma_I} \delta \tau_2 \tau_1 \bar{K} \frac{\partial a_j}{\partial x} b_i \mathrm{d}S \right] _{i,j=1}^{N_2,N_1}, \qquad \mathbf{\widetilde{I}12}_m =  \left[ \int_{\Gamma_I} \delta \tau_2 \tau_1 \widetilde{K}_T^m \frac{\partial a_j}{\partial x} b_i \mathrm{d}S \right] _{i,j=1}^{N_2,N_1}.\\
\end{align*}

To cut down the memory requirements, we take advantage of the appealing properties of our generalized low-rank approximation method in Algorithm \ref{alg1}. Given an appropriate data compressing ratio $\theta$, we adopt the generalized low-rank representation $\widetilde{\mathbf{A}}_m \approx \mathbf{U} \mathbf{V}_m^T$, and then the linear systems in (\ref{eq4.5}) yield

\begin{equation} \label{eq4.8}
\left( \bar{\mathbf{A}} +  \mathbf{U} \mathbf{V}_m^T\right)  \bm{x}_{h,\theta}^m  =   \bm{b}, \quad m = 1,2,\cdots,M.
\end{equation}

According to the Sherman-Morrison-Woodbury formula \cite{bartlett1951inverse, sherman1950adjustment, woodbury1950inverting}, we can derive the inverse of the high-dimensional sparse stiffness matrices as below

\begin{equation} 
\left(\bar{\mathbf{A}} +  \mathbf{U} \mathbf{V}_m^T\right) ^{-1}  \:  =\: \bar{\mathbf{A}}^{-1} - \bar{\mathbf{A}}^{-1} \mathbf{U} \left( \mathbf{I}_k + \mathbf{V}_m^T \bar{\mathbf{A}}^{-1} \mathbf{U}\right) ^{-1} \mathbf{V}_m^T \bar{\mathbf{A}}^{-1},\\
\nonumber
\end{equation}

\noindent for $m = 1,2,..., M$. Therefore, we can propose the following fast approach to solve the linear systems in (\ref{eq4.8}):

\begin{equation} \label{eq4.9}
\begin{aligned}
	&\bm{x}_{h,\theta}^m \: = \: \left[\bar{\mathbf{A}}^{-1} - \bar{\mathbf{A}}^{-1} \mathbf{U} \left( \mathbf{I}_k + \mathbf{V}_m^T \bar{\mathbf{A}}^{-1} \mathbf{U}\right) ^{-1} \mathbf{V}_m^T \bar{\mathbf{A}}^{-1}\right] \bm{b}\\
	& \qquad  =\: \left[ \mathbf{I}_N - \bar{\mathbf{A}}^{-1} \mathbf{U} \left( \mathbf{I}_k + \mathbf{V}_m^T \bar{\mathbf{A}}^{-1} \mathbf{U}\right) ^{-1} \mathbf{V}_m^T\right] \bar{\mathbf{A}}^{-1} \bm{b}, \\
	& \qquad  =\: \left( \mathbf{I}_N - \bar{\mathbf{A}}^{-1} \mathbf{U} \mathbf{Y}_m \mathbf{V}_m^T \right) \bar{\bm{x}}_h, \enspace\qquad m = 1,2,..., M,
\end{aligned}
\end{equation}

\noindent where $\mathbf{Y}_m := \left( \mathbf{I}_k + \mathbf{V}_m^T \bar{\mathbf{A}}^{-1} \mathbf{U}\right) ^{-1}$ is a $k \times k$ matrix, and $\bar{\bm{x}}_h := \bar{\mathbf{A}}^{-1} \bm{b}$ denotes the initial deterministic numerical solution of the unperturbed linear system. 

Let $\{\zeta_j(x)\}_{j=1}^N$ be the overall FE basis functions, then the QoI is approximated by

\begin{equation} \label{eq4.10}
\begin{aligned}
	&\hat{\bm{x}}_{h,M,\theta}(x) := \frac{1}{M} \displaystyle \sum_{m=1}^M \displaystyle \sum_{j=1}^N \bm{x}_{h,\theta}^m \:\zeta_j(x)\\
	&\qquad= \frac{1}{M} \displaystyle \sum_{m=1}^M \left[ \displaystyle \sum_{j=1}^{N_1} \phi_{h,\theta}^m a_j(x) + \displaystyle \sum_{j=1}^{N_2} \bm{u}_{1,h,\theta}^m b_j(x) + \displaystyle \sum_{j=1}^{N_2} \bm{u}_{2,h,\theta}^m b_j(x) + \displaystyle \sum_{j=1}^{N_3} \bm{p}_{h,\theta}^m c_j(x) \right].
\end{aligned}
\end{equation}

The procedure is summarized in the pseudo-code in Algorithm \ref{alg2} for solving the Stokes-Darcy model with random hydraulic conductivity and BJ condition based on the generalized low-rank matrix approximation technique, which has two main steps. The initialization process assembles the stiffness matrices $\bar{\mathbf{A}}, \widetilde{\mathbf{A}}_m$ and the load vector $\bm{b}$, which is similar to the standard MCFEM method, and the perturbed linear systems are solved via our generalized low-rank approximation of matrices and the Sherman-Morrison-Woodbury formula in the second step.

\begin{algorithm}[htbp]
\caption{A low-rank solver for the Stokes-Darcy coupled model with random hydraulic conductivity and BJ condition} 
\label{alg2}
\begin{algorithmic}[1] 
	
	\Require
	Finite element mesh $\mathcal{T}_h$ over $D$, covariance kernel $\text{Cov}(x,y)$, the truncation index $T$, the number of MC samples $M$, and data compression ratio $\theta$.
	
	\Ensure 
	QoI estimation $\hat{\bm{x}}_{h,M,\theta}$.
	
	\noindent \hspace{-2em} \textit{Step 1: Initialization}
	
	\State Construct the finite element space $V_h \subset H_0^1(D)$.
	
	\State Generate the deterministic hydraulic conductivity $\bar{K}$, and the MC realizations $\{\widetilde{K}_T^m\}_{m = 1}^M$ by the truncated KL expansion in (\ref{eq3.10}).
	
	\State Assemble the stiffness matrices $\bar{\mathbf{A}}, \{\mathbf{\widetilde{A}}_m\}_{m=1}^M$ and the load vector $\bm{b}$ by  (\ref{eq4.6}), (\ref{eq4.7}) and (\ref{eq3.14}).
	
	\State Treat the Dirichlet boundary conditions on $\bar{\mathbf{A}}$ and compute the original numerical solution $\bar{\bm{x}}_h$.
	
	\noindent \hspace{-2em} \textit{Step 2: Low-rank solver}
	
	\State Determine the reduced dimension $k = \lceil \theta N \rceil$.
	
	\State Compute the optimal solutions $\mathbf{U},\{\mathbf{V}_m\}_{m=1}^M$ of the following minimization problem by Algorithm \ref{alg1}:
	$$\mathop{\min}\limits_{ \substack{\mathbf{U},\{\mathbf{V}_m\}_{m=1}^M \in \mathbb{R}^{N \times k} \\ \mathbf{U}^T \mathbf{U} = \mathbf{I}_k}} \quad  \displaystyle \sum_{m=1}^M \Vert \mathbf{\widetilde{A}}_m - \mathbf{U} \mathbf{V}_m^T \Vert_F^2.$$
	
	\For{$m = 1,...,M$}    
	
	\State Solve the reduced linear systems to obtain the sample solutions $\bm{x}_{h,\theta}^m$ in (\ref{eq4.9}).
	
	\EndFor
	
	\State \Return the approximation $\hat{\bm{x}}_{h,M,\theta}$ of the QoI $\mathbb{E}[\bm{x}]$ in (\ref{eq4.10}).
	
\end{algorithmic}
\end{algorithm}

The efficient numerical method in Algorithm \ref{alg2} has the following several nice properties. First, our low-rank solver can significantly lessen the storage requirements for the high-dimensional stiffness matrices. Through our novel generalized low-rank matrix approximation technique in Algorithm \ref{alg1}, the perturbation matrices $\mathbf{\widetilde{A}}_m$ are replaced with the smaller-scale sketches $\mathbf{U}$ and $\mathbf{V}_m$. By using an appropriate data compression ratio $\theta$, these generalized low-rank approximations can convey enough information that we desire from the previous large-scale matrices and achieve low matrix reconstruction error. Specifically, the total storage reduction rate reads 

\begin{equation} \label{eq4.11}
\begin{aligned}
	\text{total storage reduction rate} &\:= \frac{Nk + MNk}{MNN} = \frac{k}{N} \: \left( 1 + \frac{1}{M}\right) \\
	& \:= \:\theta \: \left( 1 + \frac{1}{M}\right)  \enspace \rightarrow \: \theta, \text{ as } M \rightarrow \infty,
\end{aligned}
\end{equation}

\noindent which validates the nice data compression performance of our low-rank solver to SPDEs, especially when the number of MC samples is large. 

Second, our low-rank solver can also considerably save the computational cost for the complex coupled system. On the one hand, based on (\ref{eq4.9}), we need only 2 calculations of $k$-dimensional matrix inversion, rather than $N$-dimensional inversion in the standard MCFEM procedure. On the other hand, (\ref{eq4.9}) also indicates that $\bm{x}_{h,\theta}^m$ is perturbed from the numerical solution of the initial deterministic linear system $\bar{\bm{x}}_{h}$. Algorithm \ref{alg2} takes advantage of the perturbation form and mainly focuses on the impact of randomness in hydraulic conductivity on the numerical solution. Therefore, we compute $\bm{x}_{h,\theta}^m$ by using $\bar{\bm{x}}_{h}$ directly, rather than solving $M$ brand-new linear systems.

Moreover, both the speed-up of matrix operations and the dimensionality reduction efficacy level up as the data compression ratio $\theta$ decreases. However, a too small value of $\theta$ will result in loss of information intrinsic in the original matrices. How to balance the trade-off and determine a proper $\theta$ will be discussed in the following section.

\begin{remark}\label{rmk4.2}
In this work, the permeability field $K(\omega,x)$ is approximated by the truncated KL expansion and the stiffness matrices $\bar{A}, \widetilde{A}_m$ are directly assembled from $\bar{K}(x), \widetilde{K}(\omega,x)$ in (\ref{eq3.10}) respectively. However, our low-rank solver for the random Stokes-Darcy model is not restricted to the case where $\bar{K}(x)$ and $\widetilde{K}_T(\omega,x)$ are both given explicitly. A possible approach is to acquire the geometric mean $\bar{\mathbf{A}} := \frac{1}{M} \displaystyle \sum_{m=1}^M \mathbf{A}_m$ and compute $\widetilde{\mathbf{A}}_m := \mathbf{A}_m - \bar{\mathbf{A}} $, where the stiffness matrices $\mathbf{A}$ are given by (\ref{eq3.14}). It also serves as a common measure of data normalization, in which $\bar{\mathbf{A}}$ captures the dominant characteristics of the original coefficient matrices, while $\widetilde{\mathbf{A}}_m$ encodes the variability among samples.
\end{remark}

\begin{remark}\label{rmk4.3}
The generalized low-rank approximation of matrices makes use of similarities among the original data, and its accuracy will depend on the intrinsic correlations within the collection of stiffness matrices $\{\mathbf{A}_m\}_{m=1}^M$. Therefore, to ensure the efficacy of Algorithm \ref{alg1}, the stochastic term $\widetilde{K}(\omega,x)$ in the perturbation form (\ref{eq3.8}) needs to be a Gaussian random field and have a smooth covariance function. By Darcy's law, the conductivity coefficients in porous media are generally assumed to be Gaussian, and thus it is reasonable to adopt our effective generalized low-rank matrix approximation technique in this coupled model.
\end{remark}

\subsection{Error analysis} \label{subsec4.3}

In this subsection, we study the properties of the proposed generalized low-rank matrix approximation method in Algorithm \ref{alg1} and the low-rank solver for the Stokes-Darcy coupled system with random permeability and BJ condition in Algorithm \ref{alg2}. 

First, we revisit our generalized low-rank matrix approximation strategy, whose efficacy is usually evaluated by the root mean square reconstruction error (RMSRE) given by

\begin{equation} \label{eq4.12}
RMSRE \: := \: \sqrt{\frac{1}{M} \displaystyle \sum_{m=1}^M \Vert \widetilde{\mathbf{A}}_m - \mathbf{U} \mathbf{V}_m^T \Vert_F^2},
\end{equation}

\noindent which measures the total dissimilarities between the original matrices and their low-rank approximations. To obtain an effective generalized low-rank approximation of $\{\widetilde{\mathbf{A}}_m\}_{m=1}^M$, a relatively low RMSRE value is typically desired.

The following lemma is present in order to analyze the RMSRE value achieved by our novel generalized low-rank approximation of matrices in Algorithm \ref{alg1}.

\begin{lemma}[\cite{liang2005analytical}] \label{le4.2}
Let $\mathbf{N} \in \mathbb{R}^{N \times N}$ be a symmetric matrix and $\mathbf{U} \in \mathbb{R}^{N \times k}$ denote an arbitary matrix with orthonormal columns such that $\mathbf{U}^T \mathbf{U} = \mathbf{I}_k$. Define $\{\lambda_i(\mathbf{N})\}_{i=1}^k$ as the first $k$ largest eigenvalues of $\mathbf{N}$, then we have

$$\text{tr}(\mathbf{U}^T \mathbf{N} \mathbf{U}) \le \displaystyle \sum_{i=1}^k \lambda_i(\mathbf{N}),$$

\noindent and the equality holds if and only if $\mathbf{U} = \mathbf{P}\mathbf{Q}$, where $\mathbf{P}$ consists of the eigenvectors corresponding to the first $k$ largest eigenvalues of $\mathbf{N}$ and $\mathbf{Q} \in \mathbb{R}^{k \times k}$ could be any arbitary orthogonal matrix.
\end{lemma}

Our generalized low-rank matrix approximation technique in Algorithm \ref{alg1} has a non-iterative structure. In the following, we evaluate the optimal RMSRE value achieved by Algorithm \ref{alg1} and provide some analysis for the measurement. 

\begin{theorem} \label{th4.4}
Given the matrix $\mathbf{N} = \displaystyle \sum_{m=1}^M \widetilde{\mathbf{A}}_m \widetilde{\mathbf{A}}_m^T$, the RMSRE value obtained from the optimal solutions $\mathbf{U}, \{\mathbf{V}_m\}_{m=1}^M \in \mathbb{R}^{N \times k}$ of the optimization problem in (\ref{eq4.3}) is given by

\begin{equation}
	RMSRE(\theta) \: = \: \sqrt{ \Vert \widetilde{\mathbf{A}}_m \Vert_F^2 - \frac{1}{M} \displaystyle \sum_{i=1}^{l} \lambda_i(\mathbf{N})}.
	\nonumber
\end{equation}

\noindent Here $\{\lambda_i(\mathbf{N})\}_{i=1}^N$ is the dominant eigenvalues series of $\mathbf{N}$ and $l$ is defined as $l = \min\{\lceil \theta N \rceil, N_1\}$, with $\theta$ as the data compression ratio adopted in Algorithm \ref{alg1} and $N, N_1$ as the dimension and rank of the stochastic stiffness matrix $\widetilde{\mathbf{A}}_m$.
\end{theorem}

\begin{proof}
By Theorem \ref{th4.2} and Lemma \ref{le4.2}, we have $\mathbf{V}_m = \widetilde{\mathbf{A}}_m^T \mathbf{U}$, and the objective value of the optimization problem in (\ref{eq4.3}) is transformed into

\begin{align*}
	\displaystyle \sum_{m=1}^M \Vert \widetilde{\mathbf{A}}_m - \mathbf{U} \mathbf{V}_m^T \Vert_F^2 \:&= \:\displaystyle \sum_{m=1}^M  \text{tr}\left( \widetilde{\mathbf{A}}_m \widetilde{\mathbf{A}}_m^T\right) - \displaystyle \sum_{m=1}^{M} \text{tr}\left( \widetilde{\mathbf{A}}_m^T \mathbf{U} \mathbf{U}^T \widetilde{\mathbf{A}}_m\right)\\
	&= \:\displaystyle \sum_{m=1}^M  \text{tr}\left( \widetilde{\mathbf{A}}_m \widetilde{\mathbf{A}}_m^T \right) - \displaystyle \sum_{m=1}^{M} \text{tr}\left( \mathbf{U}^T \widetilde{\mathbf{A}}_m \widetilde{\mathbf{A}}_m^T \mathbf{U}\right)\\
	&= \:\displaystyle \sum_{m=1}^M  \text{tr}\left( \widetilde{\mathbf{A}}_m \widetilde{\mathbf{A}}_m^T\right) -  \text{tr}\left( \mathbf{U}^T \displaystyle \sum_{m=1}^{M} \left( \widetilde{\mathbf{A}}_m \widetilde{\mathbf{A}}_m^T \right)  \mathbf{U}\right)\\
	&= \: \displaystyle \sum_{m=1}^M  \text{tr}\left( \widetilde{\mathbf{A}}_m \widetilde{\mathbf{A}}_m^T \right) - \text{tr}\left( \mathbf{U}^T \mathbf{N}  \mathbf{U}\right)\\
	& = \: \displaystyle \sum_{m=1}^{M} \Vert \widetilde{\mathbf{A}}_m \Vert_F^2 - \displaystyle \sum_{i=1}^{k} \lambda_i(\mathbf{N}).
\end{align*}

\noindent The last line holds, because we assemble $\mathbf{U}$ by the eigenvectors of $\mathbf{N}$ corresponding to the first $k$ largest eigenvalues of $\mathbf{N}$ according to Algorithm \ref{alg1}. Since the first term $\displaystyle \sum_{m=1}^{M} \Vert \widetilde{\mathbf{A}}_m \Vert_F^2$ in the last line is a constant, we focus on the other term $ - \displaystyle \sum_{i=1}^{k} \lambda_i(\mathbf{N})$, which denotes the opposite number of the sum of the first $k$ dominant eigenvalues of $\mathbf{N}$. As mentioned in Section \ref{subsec4.1}, the stochastic stiffness matrix $\widetilde{\mathbf{A}}_m$ is a rank-deficient matrix and we have $\text{rank}(\widetilde{\mathbf{A}}_m)=N_1 < N$. Then, by the properties of matrix ranks, only the first $N_1$ dominant eigenvalues have nonzero values, i.e., $\lambda_i(\mathbf{N}) = 0, i = N_1+1,\cdots,N$. Accordingly, the objective function of our generalized low-rank approximation of matrices reduces to

\begin{equation}
	\begin{aligned}
		\displaystyle \sum_{m=1}^M \Vert \widetilde{\mathbf{A}}_m - \mathbf{U} \mathbf{V}_m^T \Vert_F^2 \:&= \: \displaystyle \sum_{m=1}^{M} \Vert \widetilde{\mathbf{A}}_m \Vert_F^2 - \displaystyle \sum_{i=1}^{l} \lambda_i(\mathbf{N}),
	\end{aligned}
	\nonumber
\end{equation}

\noindent where $l$ is defined as $l = \min\{k,N_1\}$. Here $k$ is the reduced dimension applied in Algorithm \ref{alg1}, and $N_1$ denotes the rank of the stochastic stiffness matrix $\widetilde{\mathbf{A}}_m$. By the definition of data compression ratio $\theta$, the coefficient $l$ could also be rewritten as $l = \min\{\lceil \theta N \rceil, N_1\}$, in which $N$ denotes the dimension of the original matrix $\widetilde{\mathbf{A}}_m$. Therefore, we have the following RMSRE value obtained by Algorithm \ref{alg1}

\begin{equation} 
	\begin{aligned}
		RMSRE \: &= \: \sqrt{\frac{1}{M} \displaystyle \sum_{m=1}^M \Vert \widetilde{\mathbf{A}}_m - \mathbf{U} \mathbf{V}_m^T \Vert_F^2}\\
		& = \:\sqrt{\frac{1}{M}\left( \displaystyle \sum_{m=1}^{M} \Vert \widetilde{\mathbf{A}}_m \Vert_F^2 - \displaystyle \sum_{i=1}^{l} \lambda_i(\mathbf{N})\right) }\\
		& = \:\sqrt{ \Vert \widetilde{\mathbf{A}}_m \Vert_F^2 - \frac{1}{M} \displaystyle \sum_{i=1}^{l} \lambda_i(\mathbf{N})}.\\
	\end{aligned}
	\nonumber
\end{equation}

\noindent The formulation above also indicates that the numerical accuracy of our generalized low-rank matrix approximation technique in Algorithm \ref{alg1} depends on the data compression ratio $\theta$.

\end{proof}

\begin{remark}\label{rmk4.4}

The reconstruction error achieved by our low-rank solver can also be measured by the cumulative energy ratio of matrix $\mathbf{N}$, which is defined as

\begin{equation}
	e(\theta) \: := \: \frac{\displaystyle \sum_{i=1}^{\lceil \theta N \rceil} \lambda_i(\mathbf{N})}{\displaystyle \sum_{i=1}^N\lambda_i(\mathbf{N})} \: = \: \frac{\displaystyle \sum_{i=1}^k \lambda_i(\mathbf{N})}{\displaystyle \sum_{i=1}^N\lambda_i(\mathbf{N})},
	\nonumber
\end{equation}

\noindent and the RMSRE formulation shown in Theorem \ref{th4.3} can be rewritten as

\begin{equation}
	RMSRE(\theta) \: = \: \sqrt{ \Vert \widetilde{\mathbf{A}}_m \Vert_F^2 - \frac{\displaystyle \sum_{i=1}^{N} \lambda_i(\mathbf{N})}{M} e(\theta)}.
	\nonumber
\end{equation}

\noindent Actually, the cumulative energy ratio $e(\theta)$ quantitatively reflects the proportion of information retained from the original stiffness matrices $\{\widetilde{\mathbf{A}}_m\}_{m=1}^M$.  The eigenvectors of a matrix represent the principal orientations of its associated linear transformation, while the eigenvalues indicate the scaling magnitude along these directions \cite{mcgivney2014svd}. Accordingly, the core characteristics from $\{\widetilde{\mathbf{A}}_m\}_{m=1}^M$ are effectively captured by selecting the first $k$ dominant eigenvectors of $\mathbf{N}$ in Algorithm \ref{alg1}. Their corresponding eigenvalues account for a considerable amount of energy from $\mathbf{N}$, which guarantees that our low-rank representation can preserve sufficient information. Therefore, the cumulative energy ratio $e(\theta)$ of the matrix $\mathbf{N}$ could be employed to determine an appropriate data compression ratio $\theta$ in our generalized low-rank matrix approximation method and low-rank solver, which will be further discussed in the next subsection. 

\end{remark}

In the following, the total approximation error of the low-rank solver in Algorithm \ref{alg2} is estimated by studying the errors from the standard MCFEM procedure and the generalized low-rank matrix approximation respectively. For notational convenience, we first make the following definitions of $\bm{X} := \underline{\bm{X}} \times Q_f$, $\bm{x} = \left( \phi_p,\bm{u}_f, p_f\right)^T $, where $\phi_p \in X_p$, $\bm{u}_f \in X_f$, $p_f \in Q_f$, and the norm in $\bm{X}$ is given by

\begin{equation}
\Vert \bm{x} \Vert_{\bm{X}} = \left( \Vert \phi_p \Vert_{\mathcal{H}^1(D_p)}^2  +  \Vert \bm{u}_f \Vert_{\mathcal{H}^1(D_f)}^2 + \Vert p_f\Vert_{\mathcal{L}^2(D_f)}^2\right) ^{1/2}.
\nonumber
\end{equation}

Next, we provide the necessary theorem below, which reveals that the numerical error by using the standard MCFEM approach is bounded by both the mesh size $h$ and the MC sample size $M$. 

\begin{theorem}[MCFEM discretization error \cite{yang2022multigrid}] \label{th4.5}

Let $V_h \in \bm{\mathcal{H}}^1(D)$ be a piecewise linear finite element space defined by a quasi-uniform triangulation mesh $\mathcal{T}_h$, and denote the QoI and MCFEM approximation of the stochastic coupled model in (\ref{eq3.2}) - (\ref{eq3.5}) by $\mathbb{E}[\bm{x}]$ and $\hat{\bm{x}}_{h,M}$ respectively. Let Assumption \ref{ass3.1} hold. Then the error bound is given by

\begin{equation}
	\Vert \mathbb{E}[\bm{x}] - \hat{\bm{x}}_{h,M}\Vert_{\bm{X}} \:\le \: C_1\left( h + \frac{1}{\sqrt{M}}\right) ,
	\nonumber
\end{equation}

\noindent where $C_1$ depends on $\bm{x}$, $h$ is the mesh size of $\mathcal{T}_h$ and $M$ is the number of MC samples.
\end{theorem}

To estimate the generalized low-rank approximation error for our low-rank solver in Algorithm \ref{alg2}, we first recall the following lemma. 

\begin{lemma}[\cite{wedin1973perturbation}] \label{le4.3}
If $\text{rank}(\mathbf{A}) = \text{rank}(\mathbf{B})$ holds for matrices $\mathbf{A}, \mathbf{B} \in \mathbb{R}^{N \times N}$, then we have

\begin{equation} 
	\Vert \mathbf{A}^{-1} - \mathbf{B}^{-1} \Vert \: \le \: \gamma \; \Vert \mathbf{A}^{-1} \Vert_2 \Vert \mathbf{B}^{-1} \Vert_2 \Vert \mathbf{A} - \mathbf{B}\Vert,
	\nonumber
\end{equation}

\noindent where $\gamma$ is a constant and $\Vert \cdot \Vert$ denotes any norm.

\end{lemma}

Owing to the conclusions from Theorem \ref{th4.4} and Lemma \ref{le4.3}, we can derive the error estimates of approximating the low-rank matrices as below. 

\begin{theorem}[Generalized low-rank matrix approximation error] \label{th4.6}

Let $\hat{\bm{x}}_{h,M,\theta}$ denote the numerical approximation of the QoI from the low-rank solver in Algorithm \ref{alg2}. Then it satisfies the following error estimate of

\begin{equation} 
	\Vert \hat{\bm{x}}_{h,M} - \hat{\bm{x}}_{h,M,\theta} \Vert_{\bm{X}} \: \le \: C_2 \frac{1}{\sqrt{M}}RMSRE(\theta),
	\nonumber
\end{equation}

\noindent where $M$ is the MC sample size, and $RMSRE$ achieved by Algorithm \ref{alg1} is related to the data compression ratio $\theta$.
\end{theorem}

\begin{proof}

For notational simplicity, we define $\mathbf{B}^m := \bar{\mathbf{A}} + \mathbf{U} \mathbf{V}_m^T$. According to Jensen's inequality, Lemma \ref{le4.3} and properties of matrix norm, the left-hand side becomes

\begin{align*}
	\Vert \hat{\bm{x}}_{h,M}& - \hat{\bm{x}}_{h,M,\theta} \Vert_{\bm{X}}^2 \:  = \: \Vert \frac{1}{M} \displaystyle \sum_{m=1}^M (u_h^m(x)) - u_{h,\theta}^m(x))\Vert_{\bm{X}}^2 \\
	& \le \: \frac{1}{M^2} \displaystyle \sum_{m=1}^M\Vert u_h^m(x)) - u_{h,\theta}^m(x) \Vert_{\bm{X}}^2 \: = \: \frac{1}{M^2} \displaystyle \sum_{m=1}^M \Vert \left( \mathbf{A}_m^{-1} - \mathbf{B}_m^{-1} \right)  \bm{b} \; \zeta(x) \Vert_{\bm{X}}^2\\
	& \le \: \frac{1}{M^2}  \Vert \bm{b} \Vert^2 \Vert \zeta(x) \Vert_{\bm{X}}^2\; \displaystyle \sum_{m=1}^M \Vert \mathbf{A}_m^{-1} - \mathbf{B}_m^{-1} \Vert_2^2 \\
	& \le \: \frac{1}{M^2} \gamma^2 \Vert \bm{b} \Vert^2 \Vert \zeta(x) \Vert_{\bm{X}}^2 \Vert \mathbf{A}_m^{-1} \Vert_2^2 \Vert \mathbf{B}_m^{-1} \Vert_2^2 \displaystyle \sum_{m=1}^M \Vert \mathbf{A}_m - \mathbf{B}_m \Vert_2^2\\
	& \le \: \frac{1}{M^2} \gamma^2 \Vert \bm{b} \Vert^2 \Vert \zeta(x) \Vert_{\bm{X}}^2 \Vert \mathbf{A}_m^{-1} \Vert_2^2 \Vert \mathbf{B}_m^{-1} \Vert_2^2 \displaystyle \sum_{m=1}^M \Vert \widetilde{\mathbf{A}}_m - \mathbf{U}\mathbf{V}_m^T \Vert_F^2\\
	& = \: \frac{1}{M} \gamma^2 \Vert \bm{b} \Vert^2 \Vert \zeta(x) \Vert_{\bm{X}}^2 \Vert \mathbf{A}_m^{-1} \Vert_2^2 \Vert \mathbf{B}_m^{-1} \Vert_2^2 \left( \frac{1}{M} \displaystyle \sum_{m=1}^M \Vert \widetilde{\mathbf{A}}_m - \mathbf{U}\mathbf{V}_m^T \Vert_F^2\right) \\
	& = \: \frac{1}{M} \gamma^2 \Vert \bm{b} \Vert^2 \Vert \zeta(x) \Vert_{\bm{X}}^2 \Vert \mathbf{A}_m^{-1} \Vert_2^2 \Vert \mathbf{B}_m^{-1} \Vert_2^2 \left( RMSRE\right)^2. \\
\end{align*}

\noindent where $\phi(\bm{x})$ is the finite element basis function. The last two lines hold due to the definition of RMSRE. Thus, we obtain  

\begin{equation}
	\begin{aligned}
		\Vert \hat{\bm{x}}_{h,M} - \hat{\bm{x}}_{h,M,\theta} \Vert_{\bm{X}} \: &\le \: \left( \Vert \hat{\bm{x}}_{h,M} - \hat{\bm{x}}_{h,M,\theta} \Vert_{\bm{X}}^2 \right)^{\frac{1}{2}}\\
		& \le \: C_2 \: \frac{1}{\sqrt{M}} \: RMSRE(\theta), \\
		\nonumber
	\end{aligned}
\end{equation}

\noindent where $C_2:=\gamma \Vert \bm{b} \Vert \Vert \zeta(x) \Vert_{\bm{X}} \Vert \mathbf{A}_m^{-1} \Vert_2 \Vert \mathbf{B}_m^{-1} \Vert_2$ is a constant. In addition, by Theorem \ref{th4.4}, the RMSRE value achieved by the novel generalized low-rank approximation of matrices in Algorithm \ref{alg1} depends on the data compression ratio $\theta$.
\end{proof}

Finally, according to the estimations from Theorem \ref{th4.5} and \ref{th4.6}, we have the following total error analysis of our low-rank solver to the Stokes-Darcy model with random hydraulic conductivity and BJ condition in Algorithm \ref{alg2}. 

\begin{theorem}[Total error estimation of Algorithm \ref{alg2}] \label{th4.7}

For any $M \in \mathbb{N}$ and $\theta \in (0,1)$, the total error of the low-rank solver in Algorithm \ref{alg2} satisfies the following estimate:

\begin{equation} 
	\Vert \mathbb{E}[\bm{x}] - \hat{\bm{x}}_{h,M,\theta} \Vert_{\bm{X}} \: = \: \mathcal{O}(h)  +  \mathcal{O}\left( \frac{1}{\sqrt{M}}\right) + \mathcal{O}\left(\frac{1}{\sqrt{M}}RMSRE(\theta)\right),
	\nonumber
\end{equation}

\noindent where $h$ is the FE mesh size, $M$ is the number of MC realizations, and RMSRE refers to the root mean square reconstruction error given by (\ref{eq4.12}).
\end{theorem}

\begin{proof}

By the triangle inequality, we have

\begin{equation}
	\begin{aligned} 
		\Vert \mathbb{E}[\bm{x}] - \hat{\bm{x}}_{h,M,\theta} \Vert_{\bm{X}} \: & \le \: \Vert \mathbb{E}[\bm{x}] - \hat{\bm{x}}_{h,M} \Vert_{\bm{X}} \: + \:\Vert \hat{\bm{x}}_{h,M} - \hat{\bm{x}}_{h,M,\theta} \Vert_{\bm{X}} \\
		& \le \: C_1 \left( h + \frac{1}{\sqrt{M}}\right)   +  C_2 \left( \frac{1}{\sqrt{M}} RMSRE(\theta)\right)  \\
		& = \: \mathcal{O}(h) \enspace + \enspace \mathcal{O}\left( \frac{1}{\sqrt{M}}\right) + \mathcal{O}\left(\frac{1}{\sqrt{M}}RMSRE(\theta)\right).
		\nonumber
	\end{aligned}
\end{equation}

\noindent According to Theorem \ref{th4.4}, the reconstruction error in the third term is related to the data compression ratio $\theta$.
\end{proof}

\section{Numerical experiments}\label{sec5}

In this section, numerical experiments on the stochastic Stokes-Darcy interface model in (\ref{eq3.2}) - (\ref{eq3.5}) are conducted to demonstrate the features of our low-rank solver in Algorithm \ref{alg2} and the theoretical conclusions. Section \ref{subsec5.1} generates the realizations of random hydraulic conductivity $\mathbb{K}$ by using the truncated KL expansion. The effectiveness of our low-rank solver is demonstrated in Section \ref{subsec5.2}. Section \ref{subsec5.3} addresses how to determine an appropriate data compression ratio $\theta$, and Section \ref{subsec5.4} validates the convergence of statistical moments of our numerical solutions when $\theta$ is properly selected.

We consider a two-dimensional domain consisting of two rectangles, where the porous media domain $D_p = (0,1) \times (0,0.5)$ and the conduit domain $D_f = (0,1) \times (-0.5,0)$ share the common interface $\Gamma_I = (0,1)\times\{0\}$. The boundaries include $\Gamma_p = \{0,1\} \times (0,0.5) \cup(0,1)\times\{0.5\}$ and $\Gamma_f = \{0,1\} \times (-0.5,0) \cup(0,1)\times\{-0.5\}$. Figure \ref{fig3.1} provides the sketch of the whole region. The physical parameters in the stochastic Stokes-Darcy model (\ref{eq3.2}) - (\ref{eq3.5}) are chosen as follows: the elevation head $z = 0$, the gravitational acceleration $g=1$, the dynamic viscosity $\nu = 1$ and the coefficient $\alpha=1$ in the BJ interface condition (\ref{ZZc}).

The source terms and boundary condition data functions are given by

\begin{equation}
\begin{aligned}
	& f_p(x) = 0, \qquad \quad\;\bm{f}_f = \bm{0},\\
	& \phi_p = 0, \qquad \qquad \:\:\,\text{on  } \Gamma_p,\\
	& \bm{u}_f = \left( 0,0\right) ^T, \qquad \text{on  }(0,1)\times\{-0.5\},\\
	& \bm{u}_f = \left( 1,0\right) ^T, \qquad \text{on  }\{0,1\} \times (-0.5,0).
\end{aligned}
\nonumber
\end{equation}

For the FEM discretization with respect to the physical space, a uniform tessellation $\mathcal{T}_h$ is employed with the grid size $h = 1/32$ and the degrees of freedom $N = 6996$. We utilize the quadratic finite elements for the primary formulation of the Darcy domain, and the Taylor-Hood elements for the Stokes domain. Gauss quadrature is applied to compute the integrals in the experiments. Specifically, we use the seven-point Gauss quadrature rule in every FE triangle of the tessellation $\mathcal{T}_h$ and the three-point Gauss quadrature rule on the interface respectively. The numerical simulations are carried out in the MATLAB R2022a software on an Apple M1 machine with 8 GB of memory.

\subsection{Realizations of the random hydraulic conductivity} \label{subsec5.1}

We first focus on the generation of MC realizations of the random hydraulic conductivity in our coupled problem, which is defined as a diagonal matrix $\mathbb{K}(\omega,x) = K(\omega,x)\mathbb{I}$. Thus, we construct the random tensor by generating $K(\omega,x)$ for $d=2$ times. The covariance function of $\mathbb{K}(\omega,x)$ is present as

\begin{equation} 
\text{Cov}(x,y)  \:= \:  \exp\left( -\frac{\vert x-y\vert^2}{0.2}\right), \quad \forall \:x,x^{\prime} \in \bar{D}_p.
\nonumber
\end{equation}

By the truncated KL expansion in (\ref{eq3.10}), the hydraulic conductivity $K(\omega,x)$ consists of the following two components

\begin{equation} 
K(\omega,x)  \:= \:\bar{K}(x) + \widetilde{K}(\omega,x),
\nonumber
\end{equation}

\noindent where the mean function is given by

\begin{equation} 
\bar{K}(x)\: = \:1, \quad \forall x \in  \bar{D}_p,
\nonumber
\end{equation}

\noindent and the perturbations are illustrated as

\begin{equation} 
\widetilde{K}(\omega,x)\:=\: \displaystyle \sum_{t=1}^{T} \sqrt{\lambda_t}\,r_t(x)\, Y_t(\omega).
\nonumber
\end{equation}

In the numerical example, $\{Y_t(\omega)\}_{t=1}^T$ are i.i.d. random variables following the standard normal distribution $\mathcal{N}(0,1)$ with the truncated interval $\left[ -3,3\right] $. By the 3-$\sigma$ rule of the standard normal distribution, approximately $99.73\%$ of the values lie within three standard deviations of the mean, i.e., in the interval of $\left[ -3,3\right] $ in this example. Therefore, the truncated distribution adopted in the simulations could not only conserve the main information and properties of the standard normal distribution but also avoid the extreme values of $Y_t$, which ensures the non-negativity of hydraulic conductivity and satisfies Assumption \ref{ass3.1}. 

Let the error tolerance of the truncated KL expansion be $\epsilon = 0.01$. Then we provide the following criterion for selecting an appropriate truncation coefficient $T$,

\begin{equation}
\rho_T\::=\: \frac{\sum\limits_{t=1}^{T}\lambda_t}{\sum\limits_{t=1}^{\infty}\lambda_t} \:\le\: 1-\epsilon,
\nonumber
\end{equation}

\noindent where $\rho_T$ is the cumulative energy ratio of the truncated KL expansion. In the simulations, we choose $T=9$ and its cumulative energy ratio reaches $\rho_9 = 99.18\%$, and Figure \ref{fig5.1} also verifies that the first 9 truncated terms could retain enough information of the whole expression of KL expansion.

\begin{figure}[htbp]

\centering
\includegraphics[width=0.9\textwidth]{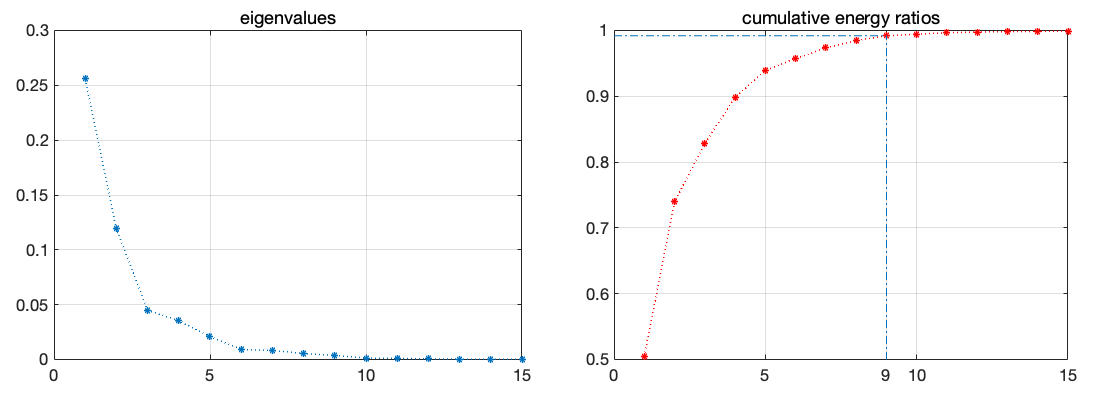}
\caption{The first 15 largest eigenvalues of the random permeability $K(\omega,x)$ (left) and their corresponding cumulative energy ratios (right).}
\label{fig5.1}
\end{figure}

We take the MC sample size $M=1800$ for the probabilistic discretization, and then $1800$ realizations of the random field in the permeability are generated in the numerical simulations. Figure \ref{fig5.2} visualizes four randomly selected realizations of $K(\omega,x)$, which exhibits the randomness of the hydraulic conductivity. 

\begin{figure}[htbp]

\centering
\includegraphics[width=0.9\textwidth]{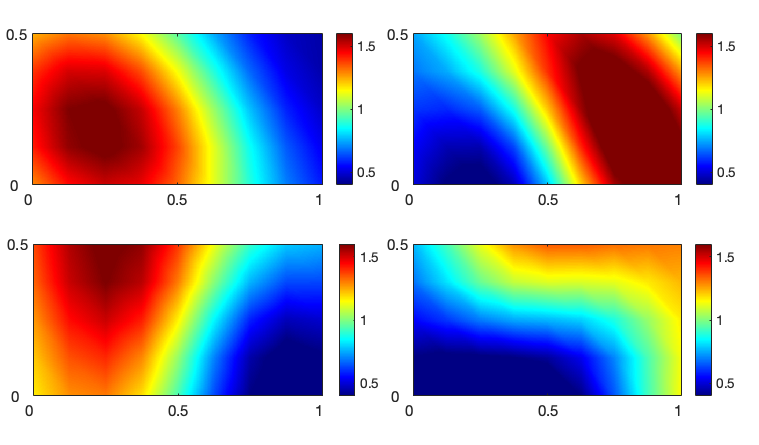}
\caption{Four randomly selected MC realizations of the hydraulic conductivity $K(\omega,x)$.}
\label{fig5.2}
\end{figure}

\subsection{Efficacy of the low-rank solver} \label{subsec5.2}

The subsection evaluates the numerical performance of the low-rank solver for the steady Stokes-Darcy interface model with random hydraulic conductivity and BJ condition in Algorithm \ref{alg2}. First, we demonstrate the numerical solutions with four randomly selected MC realizations of $K(\omega,x)$ in Figure \ref{fig5.3}, which manifests the stochastic property of the problem. Moreover, it is found that the randomness in hydraulic conductivity is transferred from the Darcy region $D_p$ to the Stokes region $D_f$ through the interface $\Gamma_I$.

\begin{figure}[htbp]

\centering
\begin{subfigure}[t]{0.45\textwidth}
	\includegraphics[width=\textwidth]{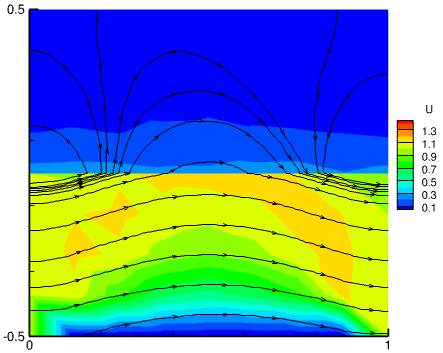}
\end{subfigure}
\hfill 
\begin{subfigure}[t]{0.45\textwidth}
	\includegraphics[width=\textwidth]{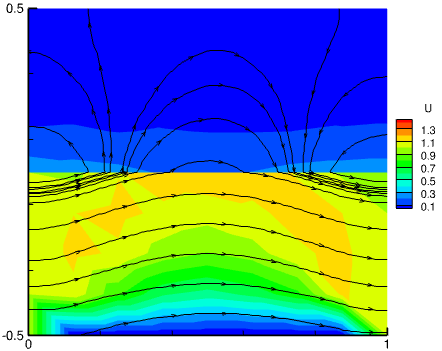}
\end{subfigure}
\\ 
\vspace{1em} 
\begin{subfigure}[t]{0.45\textwidth}
	\includegraphics[width=\textwidth]{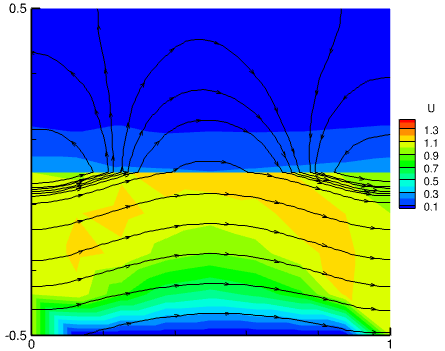}
\end{subfigure}
\hfill
\begin{subfigure}[t]{0.45\textwidth}
	\includegraphics[width=\textwidth]{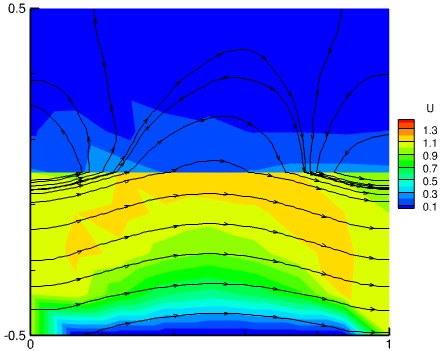}
\end{subfigure}
\caption{Four randomly selected samples of reference solution and streamlines, where color represents the velocity of flow.}
\label{fig5.3}
\end{figure}

The reference solution is generated through the standard MCFEM method with the grid size $h = 1/64$ and the number of MC realizations $M = 2000$. Figure \ref{fig5.4} compares the initial numerical solution $\bar{\bm{x}}$ with the reference numerical solution $\hat{\bm{x}}_{h,M}$, and the perturbations in the porous media domain in the right subplot presents the influence of uncertainties in the hydraulic conductivity on the Stokes-Darcy flow speed and its velocity streamlines. 

\begin{figure}[htbp] 

\begin{subfigure}[t]{0.45\textwidth}
	\includegraphics[width=\textwidth]{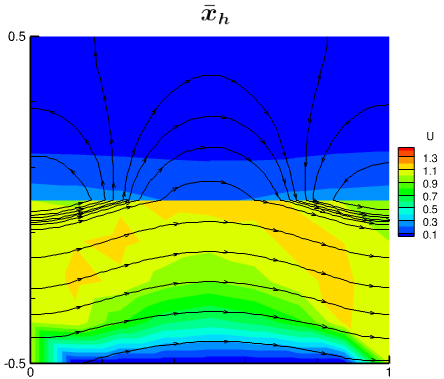}
\end{subfigure}
\hfill 
\begin{subfigure}[t]{0.45\textwidth}
	\includegraphics[width=\textwidth]{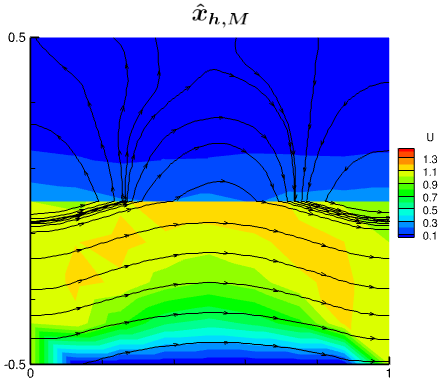}
\end{subfigure}
\caption{The initial solution $\bar{\bm{x}}_h$ and streamlines with deterministic hydraulic conductivity $\bar{K}(x)$ (left), and the reference solution $\hat{\bm{x}}_{h,M}$ and streamlines with stochastic hydraulic conductivity $\bar{K}(x) + \widetilde{K}(\omega,x)$ (right) of the steady Stokes-Darcy interface model.}
\label{fig5.4}
\end{figure}

Then we evaluate the accuracy and efficiency of the low-rank solver in Algorithm \ref{alg2} under different choices of data compression ratios. We employ six different $\theta = 100\%,70\%,50\%,30\%,10\%,5\%$ and summarize the experiment outputs in Figure \ref{fig5.5} and Table \ref{tab5.1}. When selecting a proper data compression ratio, say $30\%$ here, the low-rank solver can well capture the perturbations in the Stokes-Darcy solution caused by uncertainties in the hydraulic conductivity, produce a high-precision numerical solution, and in the meantime, has relatively low CPU elapsed time in solving large-scale linear systems. In other words, our low-rank solver can significantly reduce computational and space complexities without losing numerical accuracy, which confirms the effectiveness and validity of our low-rank solver in Algorithm \ref{alg2}.

\begin{figure}[htbp] 

\centering
\begin{subfigure}[t]{0.33\textwidth}
	\includegraphics[width=\textwidth]{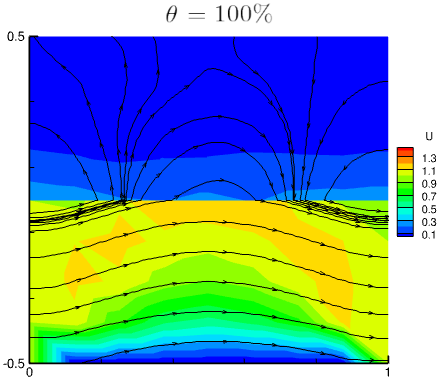}
\end{subfigure}
\hfill
\begin{subfigure}[t]{0.33\textwidth}
	\includegraphics[width=\textwidth]{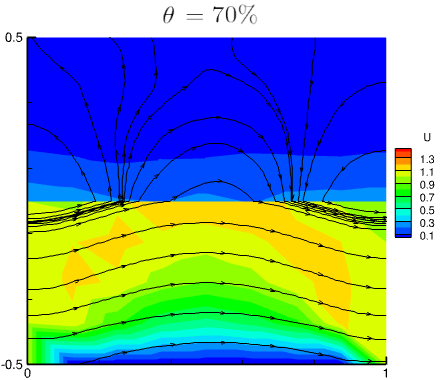}
\end{subfigure}
\hfill
\begin{subfigure}[t]{0.33\textwidth}
	\includegraphics[width=\textwidth]{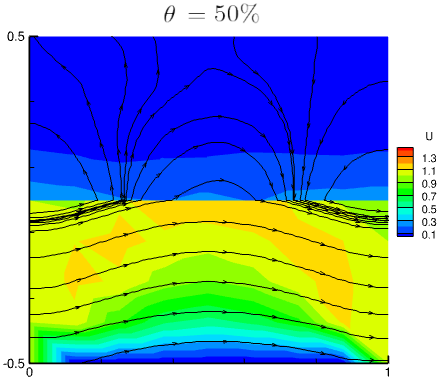}
\end{subfigure}
\\ 
\vspace{1em} 
\begin{subfigure}[t]{0.33\textwidth}
	\includegraphics[width=\textwidth]{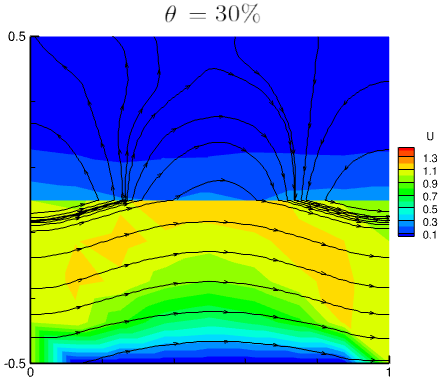}
\end{subfigure}
\hfill
\begin{subfigure}[t]{0.33\textwidth}
	\includegraphics[width=\textwidth]{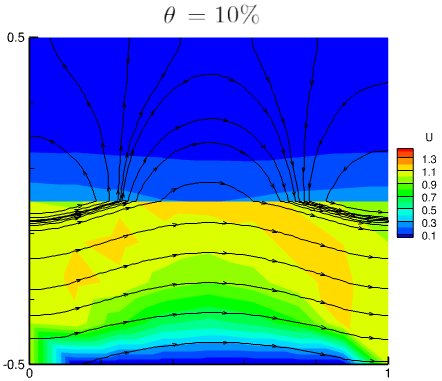}
\end{subfigure}
\hfill
\begin{subfigure}[t]{0.33\textwidth}
	\includegraphics[width=\textwidth]{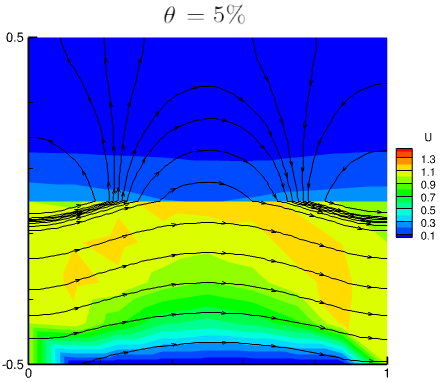}
\end{subfigure}
\caption{Numerical expectation $\hat{\bm{x}}_{h,M,\theta}$ of speed and streamlines by the low-rank solver in Algorithm \ref{alg2} under six different data compression ratios $\theta = 100\%,70\%,50\%,30\%,10\%,5\%$.}
\label{fig5.5}
\end{figure}

The outputs in Table \ref{tab5.1} also reveal the relationships among the data compression ratio $\theta$, numerical accuracy, and computational complexity. As the data compression ratio increases, our low-rank solver yields higher-precision numerical solutions. Conversely, choosing a small value of $\theta$ reduces both numerical efficiency and memory usage for the low-rank solver, which corresponds to the conclusions from Theorem \ref{th4.7} and Table \ref{tab4.1}.

In addition, Table \ref{tab5.1} presents that the errors of $\bm{u}_f, p_f$ in the free flow domain $D_f$ are notably lower than those of $\phi_p, u_p$ in the porous media domain $D_p$. The interesting observation can be attributed to the settings in the stochastic coupled model that $K(\omega,x)$ serves as the parameter in $D_p$ and on $\Gamma_I$, while the randomness is only transferred from $\Gamma_I$ to $D_f$.

\begin{table}[htbp]
\centering 
\caption{Simulation results for Algorithm \ref{alg2} about the reduced dimensions, total errors, velocity errors of Darcy and Stokes flows, and the tic-toc time for Step 2 under six different data compression ratios $\theta = 100\%,70\%,50\%,30\%,10\%,5\%$.}
\resizebox{\textwidth}{!}{
	\begin{tabular}{cccccccc}
		\hline
		Data compression ratio $\theta$ & $100\%$ & $70\%$ & $50\%$ & $30\%$ & $10\%$ & $5\%$ \\\hline
		Reduced dimension $k$ & $6996$ & $4897$ & $3498$ & $2099$ & $700$ & $350$\\
		Total error & $2.1467 \times 10^{-4}$ & $2.1472 \times 10^{-4}$ & $2.1479 \times 10^{-4}$ & $2.1485 \times 10^{-4}$ &  $3.7043 \times 10^{-1}$ & $3.7086 \times 10^{-1}$ \\
		Darcy error & $1.7211 \times 10^{-4}$ & $1.7215 \times 10^{-4}$  & $1.7223 \times 10^{-4}$ & $1.7228 \times 10^{-4}$ & $3.6357 \times 10^{-1}$ & $3.6397 \times 10^{-1}$\\
		Stokes error & $1.2831 \times 10^{-4}$ & $1.2833 \times 10^{-4}$ & $1.2834 \times 10^{-4}$ & $1.2837 \times 10^{-4}$ & $7.0972 \times 10^{-2}$ & $7.1135 \times 10^{-2}$\\
		Tic-toc time for the low-rank solver (s) & $41864.9126$ & $26199.0623$ & $18818.2782$ & $10441.1092$ & $4278.5941$ & $2352.8081$\\\hline          
\end{tabular}}
\label{tab5.1}
\end{table}

\subsection{Determination of the data compression ratio \texorpdfstring{$theta$}{Lg}} \label{subsec5.3}

How to select a suitable data compression ratio $\theta$ is a critical challenge for the numerical implementation of our low-rank solver in Algorithm \ref{alg2}. When adopting a relatively large data compression ratio, the low-rank solver yields a high-precision numerical solution, while a small $\theta$ could significantly reduce the requirements for both CPU memory and computational complexity. Therefore, it is important to balance this trade-off and make full use of the nice properties of our algorithm. In the subsection, we design a practical strategy for determining the optimal data compression ratio $\theta$.

The goal is to capture as much information from the stiffness matrices $\widetilde{\mathbf{A}}_m$ as possible through our generalized low-rank representation, thereby minimizing the reconstruction error $RMSRE(\theta)$ in (\ref{eq4.12}). To determine the optimal data compression ratio $\theta$, we can leverage the relationship established in Remark \ref{rmk4.4}, which states that $RMSRE(\theta)$ decreases as the cumulative energy ratio of the matrix $\mathbf{N} = \displaystyle \sum_{m=1}^M \widetilde{\mathbf{A}}_m \widetilde{\mathbf{A}}_m^T$ becomes larger. Therefore, it is required that the cumulative energy ratio $e(\theta)$, which is defined as below, should approach the ratio $1$:

\begin{equation}
e(\theta) \: = \: \frac{\displaystyle \sum_{i=1}^k \lambda_i(\mathbf{N})}{\displaystyle \sum_{i=1}^N \lambda_i(\mathbf{N})},
\nonumber
\end{equation}

\noindent where $\lambda_i(\mathbf{N})$ represents the $i$-th dominant eigenvalue of $\mathbf{N}$.

According to the simulation results in Figure \ref{fig5.6}, the optimal choices of the data compression ratio and the reduced dimension are $\theta = 29.72\%$ and $k = 2079$, respectively. In the numerical experiment, the matrix $\mathbf{N}$ has the dimension $N = 6996$ and the rank $N_1 = 2145$. By the property of eigenvalues, we have $\lambda_i(\mathbf{N})=0$ for $i = 2146,\cdots,6996$. The left subplot in Figure \ref{fig5.6} further implies that for $i = 2080,\cdots,2145$, the eigenvalues are of a relatively low magnitude. Specifically, we have $\lambda_{2079}(\mathbf{N})= 1.1051 \times 10^{-3}, \lambda_{2080}(\mathbf{N})= 1.2185 \times 10^{-13}$, and then the corresponding cumulative energy ratio $e(29.72\%)$ is extremely close to $1$. Therefore, when we select $\theta = 29.72\%$ and $k = 2079$, our low-rank solver can produce a high-precision numerical solution, and in the meantime, it significantly reduces the storage requirements and computational complexity associated with matrix inversion.

\begin{figure}[htbp]

\centering
\includegraphics[width=0.9\textwidth]{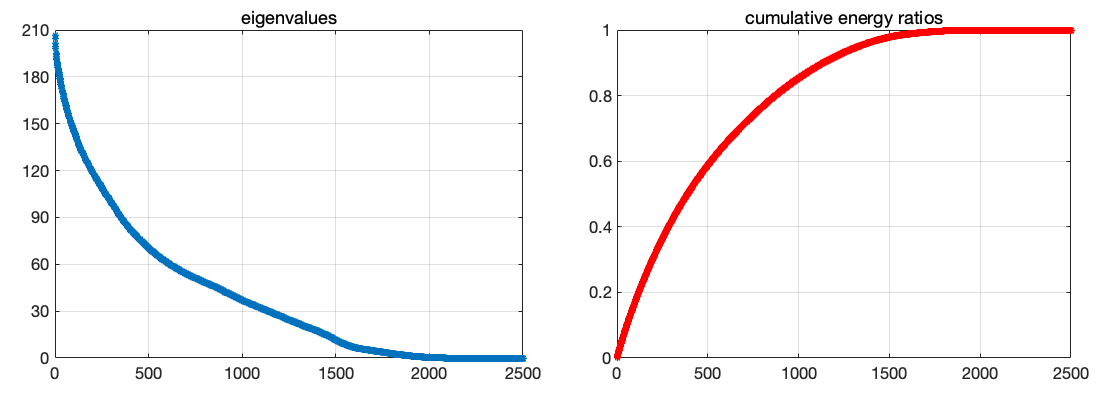}
\caption{The first 2500 largest eigenvalues of the matrix $\mathbf{N} = \displaystyle \sum_{m=1}^M \widetilde{\mathbf{A}}_m \widetilde{\mathbf{A}}_m^T$ (left) and their corresponding cumulative energy ratios (right).}
\label{fig5.6}
\end{figure}

\subsection{Convergence of statistical moments} \label{subsec5.4}

Finally, we validate the convergence of statistical moments for numerical solutions computed by our low-rank solver in Algorithm \ref{alg2}. Figure \ref{fig5.7} depicts the expectation error $\Vert \hat{\bm{x}}_{M_{ref}} - \hat{\bm{x}}_{M,\theta}\Vert_{\bm{X}}$ and variance error $\Vert \mathbb{V}[\hat{\bm{x}}_{M_{ref}}] - \mathbb{V}[\hat{\bm{x}}_{M,\theta}]\Vert_{\bm{X}}$, where the reference moments are generated using the standard MCFEM method with $h = 1/64$ and $M_{ref} = 2000$:

$$\hat{\bm{x}}_{M_{ref}} := \frac{1}{M_{ref}} \displaystyle \sum_{m=1}^{M_{ref}} \bm{x}^m, \quad \mathbb{V}[\hat{\bm{x}}_{M_{ref}}] := \frac{1}{M_{ref}} \displaystyle \sum_{m=1}^{M_{ref}} \left(\bm{x}^m - \hat{\bm{x}}_{M_{ref}}\right)^2,$$

\noindent and the expectation and variance are computed by Algorithm \ref{alg2} under the same grid size $h = 1/64$, varying sample sizes $M = 300,600,900,1200,1500$, and the data compression ratio $\theta = 30\%$:

$$\hat{\bm{x}}_{M,\theta} := \frac{1}{M} \displaystyle \sum_{m=1}^{M} \bm{x}_{\theta}^m, \quad \mathbb{V}[\hat{\bm{x}}_{M,\theta}] := \frac{1}{M} \displaystyle \sum_{m=1}^M \left(\bm{x}_{\theta}^m - \hat{\bm{x}}_{M_{ref}}\right)^2.$$

Figure \ref{fig5.7} demonstrates that both $\hat{\bm{x}}_{M,\theta}$,$\mathbb{V}[\hat{\bm{x}}_{M,\theta}]$ converge to their respective reference moments $\hat{\bm{x}}_{M_{ref}}$,$\mathbb{V}[\hat{\bm{x}}_{M_{ref}}]$ as MC sample size $M$ increases. These results confirm that our low-rank solver in Algorithm \ref{alg2} could preserve statistical accuracy when an appropriate data compression ratio $\theta$ is selected. Conversely, if adopting a too small $\theta$, the generalized low-rank matrix approximation error will dominate the total error of Algorithm \ref{alg2}, preventing the probabilistic convergence from being observed.

\begin{figure}[htbp]

\centering
\includegraphics[width=0.9\textwidth]{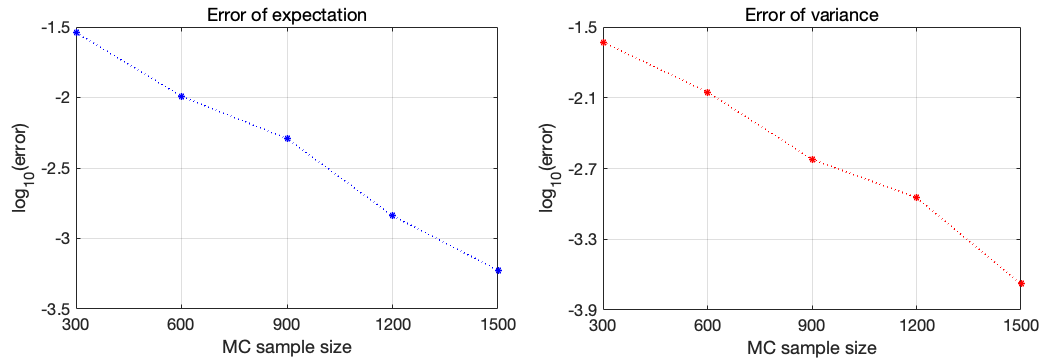}
\caption{Numerical errors of expectation $\Vert \hat{\bm{x}}_{M_{ref}} - \hat{\bm{x}}_{M,\theta}\Vert_{\bm{X}}$ (left) and variance $\Vert \mathbb{V}[\hat{\bm{x}}_{M_{ref}}] - \mathbb{V}[\hat{\bm{x}}_{M,\theta}]\Vert_{\bm{X}}$ (right), where $\hat{\bm{x}}_{M_{ref}}, \mathbb{V}[\hat{\bm{x}}_{M_{ref}}]$ are the expectation and variance of the reference solution generated by the standard MCFEM method with $M_{ref} = 2000$, and $\hat{\bm{x}}_{h,M,\theta},\mathbb{V}[\hat{\bm{x}}_{h,M,\theta}]$ denote the expectation and variance of the numerical solutions obtained by Algorithm \ref{alg2} with $\theta = 30\%$ and different numbers of MC realizations $M = 300,600,900,1200,1500$.}
\label{fig5.7}
\end{figure}

\section{Conclusions}\label{sec6}

In this work, we develop a low-rank solver for efficiently and accurately solving the Stokes-Darcy interface model with the random hydraulic conductivity and Beavers–Joseph interface condition. The randomness in the hydraulic conductivity is represented via the truncated Karhunen-Loe\.{v}e expansion. By decomposing the stochastic hydraulic conductivity and adopting a novel generalized low-rank matrix approximation method, the low-rank solver can significantly reduce the space and computational complexities from obtaining the inverse of a collection of high-dimensional perturbed stiffness matrices without losing accuracy. Error analysis is provided, and numerical experiments are established to demonstrate the feasibility and effectiveness of our low-rank solver.

%However, there are still several open questions to be answered. Firstly, future research should consider the negative effects of ill-conditioning of stiffness matrices and it is necessary to adopt measures to improve the stability of the linear systems. In addition, the applicability of our low-rank solver on unsteady or nonlinear systems, such as the time-dependent Navier-Stokes equation, needs to be further investigated. We hope that in the near future, we will be able to find answers to these questions.

\bibliographystyle{spmpsci}
\nocite{*}
%\bibliography{reference}

\end{document}